\documentclass[11pt,reqno]{amsart}
\usepackage[T1]{fontenc}
\usepackage{lmodern,amsmath,amssymb,amscd,mathtools,microtype}
\usepackage[a4paper,margin=27mm]{geometry}
\usepackage[unicode,hidelinks]{hyperref}
\usepackage{mathrsfs}
\numberwithin{equation}{section}
\newtheorem{theorem}{Theorem}[section]
\newtheorem{proposition}[theorem]{Proposition}
\newtheorem{lemma}[theorem]{Lemma}
\newtheorem{corollary}[theorem]{Corollary}
\theoremstyle{definition}
\theoremstyle{remark}\newtheorem{remark}[theorem]{Remark}
\newcommand{\C}{\mathbb C}
\newcommand{\D}{\mathbb D}
\newcommand{\Nrd}{\operatorname{Nrd}}
\newcommand{\Trd}{\operatorname{Trd}}

\newcommand{\GL}{\operatorname{GL}}
\newcommand{\SL}{\operatorname{SL}}
\newcommand{\Mat}{\operatorname{Mat}}
\newcommand{\End}{\operatorname{End}}
\newcommand{\Hom}{\operatorname{Hom}}
\newcommand{\Ad}{\operatorname{Ad}}

\newcommand{\diag}{\operatorname{diag}}
\newcommand{\rank}{\operatorname{rank}}
\newcommand{\im}{\operatorname{im}}
\newcommand{\supp}{\operatorname{supp}}
\newcommand{\Lie}{\operatorname{Lie}}
\newcommand{\Gr}{\operatorname{Gr}}

\newcommand{\cS}{\mathcal S}
\newcommand{\cD}{\mathcal D}
\newcommand{\cM}{\mathcal M}
\newcommand{\cF}{\mathcal F}
\newcommand{\cL}{\mathcal L}
\newcommand{\cR}{\mathcal R}
\newcommand{\fh}{\mathfrak h}

\newcommand{\one}{\mathbf 1}
\newcommand{\boxt}{\boxtimes}
\title[Fourier duality and twisted linear periods]
{Uniqueness of twisted linear periods}

\author{Yufeng Li}
\address{Beijing International Center for Mathematical Research, Peking University, No.5 Yiheyuan Road, Beijing 100871, China}
\email{ 2401110016@stu.pku.edu.cn}

\author{Hengfei Lu}
\address{School of Mathematical Sciences, Beihang University, 9 Nansan Street, Shahe Higher Education Park, Changping, Beijing, 102206, China}
\email{luhengfei@buaa.edu.cn}

\author{Junyan Xiao}
\address{Beijing International Center for Mathematical Research, Peking University, No.5 Yiheyuan Road, Beijing 100871, China}
\email{ jyxiao25@stu.pku.edu.cn}

\date{\today}
\subjclass[2020]{22E50}
\keywords{Linear periods, quaternion algebras, equivariant distributions,
Fourier transform, Harish--Chandra descent}

\begin{document}
\raggedbottom
\begin{abstract}
Let $F$ be a non-Archimedean local field of characteristic zero, and let
$\D$ be either $F$ or a quaternion division algebra over $F$.
We prove that every irreducible smooth representation of
$\GL_{p+q}(\D)$ has at most one linear period transforming under any
prescribed character of $\GL_p(\D)\times\GL_q(\D)$.
\end{abstract}
\maketitle
\setcounter{tocdepth}{1}\tableofcontents

\section{Introduction}
Let $F$ be a non-Archimedean local field of characteristic zero, and
write $\nu=|\cdot|_F$. Let $\D$ be either $F$ or a quaternion division
algebra over $F$, and put
\[
 d=\deg_F\D\in\{1,2\},\qquad \dim_F\D=d^2.
\]
Write $x\mapsto\bar x$ for the identity on $F$ or the canonical
quaternionic involution, respectively, and set
$A^\dagger=(\overline{a_{ji}})$ for a matrix $A=(a_{ij})$.
The symbols $\Nrd$ and $\Trd$ denote reduced norm and reduced trace
on the relevant matrix algebra over $F$. When $\D=F$, they are the
ordinary determinant and trace and $x\mapsto\bar{x}$ becomes the identity map. All representations and distributions
have complex coefficients. Fix integers $p,q\geq1$, and set
\[
 G=\GL_{p+q}(\D),\qquad H=\GL_p(\D)\times\GL_q(\D).
\]
Let
\[
 \eta=\alpha\circ\Nrd\boxt\beta\circ\Nrd
\]
be an arbitrary smooth character of $H$, where
$\alpha,\beta:F^\times\to\C^\times$ are smooth characters.
A linear period on a smooth representation $\pi$ is an algebraic linear functional $\ell$ satisfying
\[
\ell(\pi(h)v)=\eta(h)\ell(v)\qquad(h\in H).
\]

\begin{theorem}\label{thm:main}
For every irreducible smooth representation $\pi$ of $G$ and every smooth
character $\eta$ of $H$,
\[
\dim\Hom_H(\pi,\eta)\leq1.
\]
\end{theorem}

In the split case, Jacquet--Rallis proved untwisted uniqueness
for arbitrary block sizes (see \cite[Theorem~1.1]{JR}).
Chen--Sun established balanced twisted uniqueness under a
goodness condition on the character
(see \cite[Theorems~B and~C]{CS}).
Lu treated adjacent block sizes under a corresponding character
hypothesis and proved unrestricted uniqueness when
$\min(p,q)=1$ (see \cite[Theorems~1.1 and~1.3]{Lu}).
The two-factor and one-factor twisting conventions are related by
\[
 \Hom_H\bigl(\pi,\alpha\circ\Nrd\boxt\beta\circ\Nrd\bigr)
 =\Hom_H\bigl(\pi\otimes(\beta^{-1}\circ\Nrd),
             (\alpha\beta^{-1})\circ\Nrd\boxt\one\bigr).
\]
Both sides are subspaces of the algebraic dual of the same
underlying vector space.

For general division algebras,
Anandavardhanan--Lu--Matringe--S\'echerre--Yang
proved untwisted uniqueness for all irreducible representations
when $|p-q|\leq1$, and for irreducible unitary representations
with arbitrary block sizes
(see \cite[Theorem~2.5 and Appendix~A]{ALMSY}).
Lu \cite{LuDA} formulated the general arbitrary-character
problem and obtained further results for certain classes of
representations.

Theorem~\ref{thm:main} establishes twisted uniqueness for all
irreducible smooth representations over $F$ and its quaternion
division algebra. In the Archimedean setting, Li--Xiao
proved the analogous result for irreducible Casselman--Wallach
representations over $\mathbb R$, $\mathbb C$, and $\mathbb H$
(see \cite[Theorem~1.1]{LX}).
Both results allow arbitrary block sizes and require no
goodness condition on the character.

On the global side, Xue--Zhang \cite{XZ} used a relative
trace formula to study twisted linear periods, establishing
a relative fundamental lemma and transfer of orbital integrals.

We will use a symmetry theorem
for distributions to give the proof of Theorem~\ref{thm:main}  

For an $F$-analytic manifold $X$ equipped with the action of $G$, set $\cS(X)=C_c^\infty(X)$ and
$\cD(X)=\cS(X)^*$. The action on test functions is
\begin{equation}\label{eq:convention}
\rho(l)f(x)=f(l^{-1}x),
\end{equation}
where $f\in S(X)$ and $l\in G$.
For an involution $s$, write $s_*T(f)=T(f\circ s)$.

\begin{theorem}\label{thm:group}
Every $T\in\cD(G)$ satisfying
\begin{equation}\label{eq:biequiv}
T\bigl(f(h_l^{-1}\,\cdot\,h_r)\bigr)
 =\eta(h_l)\eta(h_r)^{-1}T(f)
\end{equation}
for any $h_l,h_r\in H$ and $f\in S(G)$  is fixed by the anti-involution $g\mapsto g^{\dagger}$.
\end{theorem}

Together with the Gelfand--Kazhdan criterion (see \cite[Theorem 2.21]{JLX}), Theorem \ref{thm:main} will follow from Theorem \ref{thm:group} immediately.



The proof of Theorem~\ref{thm:group} uses two linear symmetries.
For pairs $(A,B)\in\Mat_r(\D)^2$ of square matrices, with action
\begin{equation}\label{eq:intro-pair}
(g,h)(A,B)=(gAh^{-1},gBh^{-1}),
\end{equation}
we employ the partial Fourier transform to establish the componentwise adjoint symmetry
for every character $\kappa(\Nrd g/\Nrd h)$ (see Theorem~\ref{thm:square}). The key argument compares
the modular condition on orbits of singular pairs with its partial Fourier
transform. For rectangular pairs with action
\begin{equation}\label{eq:intro-recip}
(g,h)(X,Y)=(gXh^{-1},hYg^{-1}),
\end{equation}
we will use the Haar Radon transform to show that every invariant distribution
is fixed by $(X,Y)\mapsto(Y^{\dagger},X^{\dagger})$ (see Theorem~\ref{thm:rectangular}). A proper incidence correspondence
first gives the support bound in Proposition~\ref{prop:image-rank}.
The Haar Radon transform then reduces the symmetry to
orthogonal-complement duality on a Grassmannian. Finally applying the generalized Harish--Chandra descent, Theorem \ref{thm:group} follows by combining Theorem \ref{thm:square} and Theorem \ref{thm:rectangular}.

The paper is organized as follow. Section $2$ introduces the necessary notation and preliminary results. In Section $3$, we compute the modular characters of the stabilizers associated with singular elements and apply the Bernstein localization principle to eliminate the contributions from irrelevant elements. Section~\ref{sec:square} is devoted to establishing the symmetry property with respect to a nontrivial character in the balanced case $(p=q=r)$. Section~\ref{sec:padding} addresses the corresponding symmetry property in the rectangular case. In Section~\ref{sec:group-descent}, we analyze the normal representations and distinguish two cases: the unbalanced factor carries a trivial character, while a nontrivial character occurs only on a balanced factor. Combining the two linear results with the generalized Harish–Chandra descent, we establish Theorem \ref{thm:group}. Finally, we deduce Theorem \ref{thm:main}.

\subsection*{Acknowledgements} The second named author was partially supported by NSFC 12301031.
The first and third named authors would like to thank
their advisor, Jun Yu, for his guidance
and Jialiang Zou for his valuable help.
They were partially supported by
NSFC Grants 12431001 and 11971036.
The authors also acknowledge the assistance of
 ChatGPT6 Astra in the development and
refinement of some of the proofs.

\section{Notation and preliminaries}\label{sec:prelim}
All matrix maps act on right $\D$-vector spaces of columns. Their
kernels and images are right subspaces; row spaces are left
$\D$-subspaces. Matrix ranks and the dimensions of these subspaces
are taken over $\D$. Lie algebras and distribution spaces are regarded
as spaces over $F$. We use the same symbol for an algebraic $F$-group
and its group of $F$-points when no confusion is possible.
For a right $\D$-vector space $V$, write $\Gr_k(V)$ for its
Grassmannian of $k$-dimensional right subspaces. Write
$\Gr_k^\ell(\D^n)$ for the corresponding left Grassmannian.
These are projective $F$-varieties. For example, $\Gr_k(V)$ is the
closed locus of right $\D$-stable $F$-subspaces of dimension $d^2k$
in the ordinary $F$-Grassmannian. When using distributions, we mean
their $F$-rational points.

Set
\begin{equation}\label{eq:lambda}
 \lambda(g)=\nu(\Nrd g)^d,\qquad g\in\GL_r(\D),
\end{equation}
for every $r$. For $g\in\GL_m(\D)$ and $h\in\GL_n(\D)$,
\begin{equation}\label{eq:matrix-Jacobian}
 \left|\det_F\bigl(X\mapsto gXh^{-1};\Mat_{m\times n}(\D)\bigr)\right|_F
 =\lambda(g)^n\lambda(h)^{-m}.
\end{equation}
Indeed, after extension to a splitting field, the matrix space becomes
$\Mat_{dm\times dn}$ and the algebraic determinant is
$\Nrd(g)^{dn}\Nrd(h)^{-dm}$. This identity descends to $F$.
In particular, $\lambda(tI_r)=\nu(t)^{d^2r}$ for $t\in F^\times$.

We use the standard structure of central simple algebras and the
Skolem--Noether theorem; see \cite{Voight}. The reduced norm gives
$\GL_r(\D)^{\mathrm{ab}}\simeq F^\times$.
For completeness, in the quaternion division case its surjectivity
follows from the unramified quadratic presentation
$\D=E\oplus E j$, with $j^2$ a uniformizer: norms from
$\mathcal O_E^\times$ exhaust $\mathcal O_F^\times$, and
$\Nrd(j)$ has valuation one. If $x\in\D^\times$ has norm one,
choose a quadratic subfield containing $x$ (any quadratic subfield
when $x$ is central). Hilbert's theorem~90 gives
$x=b/\bar b$ in that subfield. An element $u\in\D^\times$
inducing its nontrivial automorphism then gives $x=[b,u]$.
Row elimination reduces the assertion for $\GL_r(\D)$ to this
assertion for $\D^\times$. The norm map is open, so the same
factorization holds for smooth characters.

Fix a nontrivial continuous unitary additive character
$\psi:F\to\C^\times$. On $\D$ choose self-dual Haar measure for
$(x,y)\mapsto\psi(\Trd(xy))$, and on matrix spaces use its
product measure. The trace pairing is nondegenerate, and
\begin{equation}\label{eq:reduced-identities}
 \Trd(XY)=\Trd(YX),\qquad
 \Trd(Z^\dagger)=\Trd(Z),\qquad
 \Nrd(g^\dagger)=\Nrd(g).
\end{equation}
Here the first identity also applies to rectangular $X,Y$ whose
products are square. These identities follow after splitting the
matrix algebras. Regarding $g$ as a linear operator, then its adjoint operator is given by $g^\dagger$. The map $g\mapsto g^\dagger$ preserves the chosen Haar measures.

An $l$-space is a locally compact, totally disconnected Hausdorff
space with a countable basis. All spaces of rational points below are
$l$-spaces. For an $l$-space $X$, write
\[
\cS(X)=C_c^\infty(X),\qquad \cD(X)=\cS(X)^*.
\]
More generally, if $\mathcal E\to X$ is a finite-rank locally constant
vector bundle, set
\[
\cS(X,\mathcal E)=\Gamma_c^\infty(X,\mathcal E),\qquad
\cD(X,\mathcal E)=\cS(X,\mathcal E^\vee)^*.
\]
If $a$ is an automorphism of an $l$-space, we use the pushforward
convention $(a_*T)(f)=T(f\circ a)$. For a group $K$ acting on $X$,
$\cD(X)^{K,\chi}$ denotes the distributions satisfying
$T(f(g^{-1}\,\cdot))=\chi(g)T(f)$ for every $g\in K$ and $f\in S(X)$.

Let $E$ be a finite-dimensional $F$-vector space with a nondegenerate
symmetric or alternating bilinear form $B$. 
Define
\begin{equation}\label{eq:Fourier-conventions}
 (\mathcal F_Bf)(x)=\int_E f(y)\psi(B(x,y))\,dy
\end{equation}
where the fixed Haar measure $dy$ is  self-dual
with respect to $\psi$,
such that $\mathcal F_B^2f=f(-\,\cdot)$ for symmetric $B$, and
$\mathcal F_B^2f=f$ for alternating $B$. On distributions we set
$(\mathcal F_BT)(f)=T(\mathcal F_Bf)$. The pairings and variables
for partial Fourier transforms are specified where they occur.

For a subspace $W\subset E$, write
$W^\perp=\{y\in E:B(w,y)=0\text{ for all }w\in W\}$.
A Haar measure $\mu_W$ determines the quotient measure on $E/W$;
let $\mu_{W^\perp}$ be its $\psi$-dual measure under the pairing
induced by $B$. Viewing these measures as distributions on $E$, we have
\begin{equation}\label{eq:Fourier-subspace}
 \mathcal F_B\mu_W=\mu_{W^\perp}.
\end{equation}
This correspondence is linear on the lines spanned by Haar measures,
and applying it twice returns the original measure.

More generally, let $B:E\times E'\to F$ be a nondegenerate bilinear
pairing between finite-dimensional $F$-vector spaces, and fix a Haar
measure $dy$ on $E'$. For
$f\in\cS(E')$ and $T\in\cD(E)$, set
\[
 (\mathcal F_Bf)(x)=\int_{E'}f(y)\psi(B(x,y))\,dy,
 \qquad (\mathcal F_BT)(f)=T(\mathcal F_Bf).
\]
Suppose that a group $K$ acts linearly on $E$ and $E'$ by $a_g$ and $b_g$,
respectively, with $B(a_gx,b_gy)=B(x,y)$ for every $g\in K$.
Then $\nu(\det_F b_g)=\nu(\det_F a_g)^{-1}$. The substitution $y=b_gz$ gives
\begin{equation}\label{eq:Fourier-covariance}
\begin{aligned}
 \mathcal F_B(f\circ b_g^{-1})(x)
 &=\nu(\det_F b_g)\int_{E'}f(z)\psi(B(a_g^{-1}x,z))\,dz\\
 &=\nu(\det_F a_g)^{-1}(\mathcal F_Bf)(a_g^{-1}x).
\end{aligned}
\end{equation}
Consequently, if $T$ is $\chi$-equivariant for the action $a_g$, then
\begin{equation}\label{eq:Fourier-equivariance}
 (\mathcal F_BT)(f\circ b_g^{-1})
 =\chi(g)\nu(\det_F a_g)^{-1}(\mathcal F_BT)(f).
\end{equation}
Thus $\mathcal F_BT$ is equivariant for the action $b_g$ with character
$\chi'(g)=\chi(g)\nu(\det_F a_g)^{-1}$.
The same calculation applies to a partial Fourier transform from
$V\times E$ to $V\times E'$, with the same action on the untransformed
factor $V$ on both sides. The determinant factor involves only the
action on $E$. The pairing and the actions on both sides will be
specified in each application.

We use the standard calculus of distributions on $l$-spaces as in
\cite[Section~1]{BZ}. In particular, for a closed subset $Z\subset X$
we identify distributions on $Z$ with distributions on $X$ supported
on $Z$.

We shall repeatedly use Bernstein's localization principle
(see \cite[Section~1.4]{Ber}). If an $l$-group $L$ acts continuously on $X$,
$b:X\to Y$ is a continuous $L$-invariant map to an $l$-space, and
$\chi$ is a character of $L$, then
\begin{equation}\label{eq:localization}
\cD(b^{-1}(y))^{L,\chi}=0\quad \mbox{for all  } y\in Y
\quad\Longrightarrow\quad
\cD(X)^{L,\chi}=0.
\end{equation}

\subsection{The modular convention and descent}
For a smooth algebraic group $S$, write
\begin{equation}\label{eq:delta}
\delta_S(s)=\bigl|\det_F(\Ad(s)|_{\Lie S})\bigr|_F.
\end{equation}
For a left Haar measure $dx$ on $S$, our right-translation convention is
\[
\int_S a(xs)\,dx=\delta_S(s)\int_S a(x)\,dx.
\]
Thus, for a parabolic subgroup, $\delta_S$ is the absolute determinant
on the Lie algebra of its unipotent radical.

Let us recall the following result for the Frobenius descent.
\begin{lemma}[
{\cite[Section~1.5]{Ber} or \cite[Theorem~2.5.7]{AG}}]
\label{lem:frob}
Let $L$ be the group of rational points of a smooth algebraic group,
assume that $L$ is unimodular, and let $S\subset L$ be the group of
rational points of a smooth algebraic subgroup. Let $\chi$ be a smooth
character of $L$, and let $W$ be an
$l$-space with a continuous $S$-action, and let
$\mathcal E\to W$ be a finite-rank locally constant $S$-equivariant
vector bundle. Then
\begin{equation}\label{eq:fiber-character}
\cD\bigl(L\times^S W,\,L\times^S\mathcal E\bigr)^{L,\chi}
 \simeq
\cD(W,\mathcal E)^{S,\chi|_S\delta_S}.
\end{equation}
In particular, for the trivial line bundle,
\[
\cD(L\times^S W)^{L,\chi}
 \simeq \cD(W)^{S,\chi|_S\delta_S}.
\]
Taking $W$ to be a point, a nonzero $\chi$-equivariant distribution
on $L/S$ exists only if
\begin{equation}\label{eq:orbit-condition}
\chi|_S=\delta_S^{-1},
\end{equation}
and when this condition holds its space is one-dimensional.
\end{lemma}
Furthermore, Aizenbud-Gourevitch have developed the following Generalized Harish--Chandra descent.
\begin{theorem}
{\cite[Theorem~3.1.1]{AG}}
\label{thm:HC-descent}
Let $K$ be a reductive algebraic group, not necessarily connected,
acting algebraically on a smooth affine variety $X$ over $F$.
Let $\chi$ be a smooth character of $K(F)$.
For $x\in X(F)$ with closed algebraic orbit $K\cdot x$, let
\[
 N_x=T_xX/T_x(K\cdot x),
\]
with the action induced by the stabilizer $K_x$.
If
\[
 \cD(N_x(F))^{K_x(F),\chi|_{K_x(F)}}=0
\]
for every such $x$, then
$\cD(X(F))^{K(F),\chi}=0$.
\end{theorem}

We also use the standard compatibility of proper pushforward with
restriction; see \cite[\S2.3]{CHLR}. If $\pi:Y\to Z$ is a continuous
proper map of $l$-spaces, $U\subset Z$ is open, $Y_U=\pi^{-1}(U)$, and
$\pi_U=\pi|_{Y_U}$, then
\begin{equation}\label{eq:proper-restriction}
(\pi_*T)|_U=(\pi_U)_*(T|_{Y_U})
\qquad(T\in\cD(Y)).
\end{equation}
In particular, if $\pi_U$ is an isomorphism, then
\[
(\pi_U)_*(T|_{Y_U})=0
\iff
T|_{Y_U}=0.
\]

The following result is due to Rosenlicht.
\begin{theorem}[
{\cite[Theorem~2]{Rosenlicht} or \cite[Theorem~1.1(a)]{BGR}}]
\label{thm:rosenlicht}
Let an algebraic group $\mathbf L$ act on an irreducible $F$-variety $Y$.
There exist an $\mathbf L$-invariant dense open subset
$Y^\circ\subset Y$, an $F$-variety $Z$, and an
$\mathbf L$-invariant morphism
\[
b:Y^\circ\longrightarrow Z
\]
whose nonempty geometric fibers are single $\mathbf L$-orbits.
\end{theorem}
By Noetherian induction, Theorem~\ref{thm:rosenlicht} yields a finite
$\mathbf L$-invariant stratification
\begin{equation}\label{eq:rosenlicht-stratification}
Y=\bigsqcup_{i=1}^N Y_i,
\qquad
b_i:Y_i\longrightarrow Z_i,
\end{equation}
where each $Y_i$ is locally closed and each $b_i$ is
$\mathbf L$-invariant with nonempty geometric fibers consisting
of single $\mathbf L$-orbits.

\begin{proposition}
\label{prop:orbitwise-localization}
Let $L=\GL_m(\D)\times\GL_n(\D)$ act on $\Mat_{m\times n}(\D)^2$ by
\[
(g,h)(A,B)=(gAh^{-1},gBh^{-1}),
\]
and let $X\subset\Mat_{m\times n}(\D)^2$ be an $L$-invariant locally
closed subset. Let $\chi$ be a smooth character of $L$, and for
$x\in X$ let $L_x=\{\ell\in L:\ell\cdot x=x\}$ be its stabilizer.
If $\chi|_{L_x}\neq\delta_{L_x}^{-1}$ for every $x\in X$, then
$\cD(X)^{L,\chi}=0$.
\end{proposition}
\begin{proof}
For this matrix action, any two $F$-points in the same geometric
orbit already lie in the same $L$-orbit.
Indeed, let $(A,B)$ and $(A',B')$ be such points. The equations
\[
 gA=A'h,\qquad gB=B'h
\]
define an $F$-linear subspace of $\Mat_m(\D)\oplus\Mat_n(\D)$
in the unknowns $(g,h)$.
Geometric equivalence gives a solution over $\overline F$ with
$\Nrd g\Nrd h\ne0$. Hence the polynomial $(g,h)\mapsto\Nrd g\Nrd h$
does not vanish identically on this solution space.
Since $F$ is infinite, it is nonzero at some $F$-point of the
solution space. This gives $g\in\GL_m(\D)$ and $h\in\GL_n(\D)$
with $(A',B')=(gAh^{-1},gBh^{-1})$.

Apply \eqref{eq:rosenlicht-stratification} to the algebraic action of
$\mathbf L$ associated to $L$ on $\Mat_{m\times n}(\D)^2$, and intersect
the strata with $X$. By the geometric fiber property and the preceding
argument, for $x,x'\in X\cap Y_i(F)$ one has
\[
b_i(x')=b_i(x)
\quad\Longleftrightarrow\quad
x'\in L\cdot x.
\]
Since $X$ is $L$-invariant, every nonempty fiber of
$b_i|_{X\cap Y_i(F)}$ is a single orbit $L\cdot x\simeq L/L_x$.
Hence Lemma~\ref{lem:frob} gives
$\cD(L/L_x)^{L,\chi}=0$ whenever
$\chi|_{L_x}\neq\delta_{L_x}^{-1}$. The Bernstein localization on each
stratum yields $\cD(X\cap Y_i(F))^{L,\chi}=0$ for $1\le i\le N$, and
hence $\cD(X)^{L,\chi}=0$.
\end{proof}

\subsection{Conjugation symmetry}
\begin{lemma}\label{lem:rational-similarity}
Let $\mathcal A$ be a finite product of central simple algebras over
finite extensions of $F$. If $\jmath$ is an $F$-linear
anti-automorphism of $\mathcal A$ fixing its center pointwise, then
$x$ and $\jmath(x)$ are conjugate under $\mathcal A^\times$ for every
$x\in\mathcal A$.
\end{lemma}
\begin{proof}
After extension to $\overline F$, every factor is a full matrix
algebra and $\jmath$ is an inner automorphism followed by transpose.
A matrix is similar to its transpose, so the $F$-linear space
\[
 V_x=\{u\in\mathcal A:u x=\jmath(x)u\}
\]
contains an invertible element over $\overline F$. The determinant
of left multiplication by $u$ on the $F$-space $\mathcal A$ is a
polynomial on $V_x$ that is therefore not identically zero. Since
$F$ is infinite, it is nonzero at an $F$-point. That point gives
the required conjugating element.
The same argument shows that two $F$-points in a common geometric
conjugacy orbit are already conjugate under $\mathcal A^\times$.
\end{proof}

\begin{theorem}\label{thm:conjugation}
Let
\[
 \mathcal A=\prod_i\Mat_{m_i}(D_i),
\]
where each $D_i$ is either a finite extension $E_i/F$ or a quaternion
division algebra over $E_i$. Let $\jmath$ be an $F$-linear
anti-automorphism of $\mathcal A$ fixing its center pointwise.
For every conjugation-invariant, $\jmath$-stable locally closed
subset $X\subset\mathcal A$, every
$\mathcal A^\times$-invariant distribution on $X$ is fixed by
$\jmath$. In particular, every conjugation-invariant distribution
on $\Mat_m(\D)$ is fixed by the anti-involution $A\mapsto A^\dagger$.
\end{theorem}
\begin{proof}
On each matrix factor choose transpose or conjugate transpose,
as appropriate, and denote the product involution by $\jmath_0$.
The Skolem--Noether theorem shows that $\jmath\jmath_0$ is inner. It suffices
therefore to treat $\jmath_0$; the set $X$ is also
$\jmath_0$-stable.

Apply the Rosenlicht stratification to the conjugation on $\mathcal A$.
Intersect each invariant open stratum with its $\jmath_0$-image
and continue by Noetherian induction. This gives
$\jmath_0$-stable strata and invariant quotient maps whose fibers
on $F$-points are single conjugacy orbits, by
Lemma~\ref{lem:rational-similarity}. The quotient maps are
$\jmath_0$-invariant because $\jmath_0$ preserves each geometric
orbit.

On a single orbit, Lemma~\ref{lem:frob} shows that the space of
invariant distributions has dimension at most one. If nonzero,
it is spanned by a positive invariant measure. Its pushforward
under $\jmath_0$ is a positive multiple of itself; since
$\jmath_0^2=1$, that multiple is one. Thus there is no invariant,
$\jmath_0$-odd distribution on any fiber. The Bernstein localization,
applied to the group generated by conjugation and $\jmath_0$
with sign character on $\jmath_0$, gives the assertion on every
stratum and then on $X$. This finishes the proof.
\end{proof}

\begin{corollary}\cite[Theorem 3.1]{Ragh}\label{cor:contragredient}
For every irreducible smooth representation $\pi$ of $\GL_r(\D)$,
\[
 \pi^\vee\simeq\pi\circ\theta,\qquad
 \theta(g)=(g^\dagger)^{-1}.
\]
\end{corollary}
\begin{proof}
The anti-involution $g\mapsto g^\dagger$ normalizes conjugation, squares to the identity, and
preserves every conjugacy class by Lemma~\ref{lem:rational-similarity}.
The contragredient identity therefore follows from
\cite[Theorem~8.2.1]{AG}.
\end{proof}

\section{The modular obstruction for singular pairs}\label{sec:pairs}
Set
\[
\cM_{m,n}=\Mat_{m\times n}(\D)\oplus\Mat_{m\times n}(\D),\qquad
L_{m,n}=\GL_m(\D)\times\GL_n(\D),
\]
with action \eqref{eq:intro-pair}. For $(A,B)\in\cM_{m,n}$, define
its \emph{normal rank}, with $t$ central, by
\[
\operatorname{nrank}(A,B)=\rank_{\D(t)}(A+tB).
\]
Here $\D(t)=\D\otimes_F F(t)$ is a division algebra, as is seen
by embedding it in the division ring $\D((t))$. A full-rank
square pencil has $\Nrd(A+tB)\ne0$ in $F[t]$, so some central
scalar specialization is invertible.
We call $(A,B)$ \emph{regular} if
$\operatorname{nrank}(A,B)=\min(m,n)$, and \emph{singular} otherwise.
If $m=n$, the pair is singular if and only if every scalar combination
$sA+tB$, with $s,t\in F$, is a singular matrix.

\subsection{Kronecker canonical form}
For $r,s\geq0$, write $\mathbf I_r$ for the $r\times r$ identity matrix
and $\mathbf 0_{r\times s}$ for the $r\times s$ zero matrix. For $a\geq0$,
set
\[
A_a=\begin{bmatrix}\mathbf I_a&\mathbf 0_{a\times1}\end{bmatrix},
\qquad
B_a=\begin{bmatrix}\mathbf 0_{a\times1}&\mathbf I_a\end{bmatrix},
\qquad
P_a=(A_a,B_a),
\]
and for $b\geq0$ set
\[
Q_b=
\left(
\begin{bmatrix}\mathbf I_b\\ \mathbf 0_{1\times b}\end{bmatrix},
\begin{bmatrix}\mathbf 0_{1\times b}\\ \mathbf I_b\end{bmatrix}
\right)
=(A_b^{\dagger},B_b^{\dagger}).
\]

We have the following Kronecker decomposition (see \cite{Gantmacher}).
\begin{theorem}
\label{thm:kronecker}
Let $(A,B)\in\cM_{m,n}$ and put $\rho=\operatorname{nrank}(A,B)$.
Then $(A,B)$ is $L_{m,n}$-equivalent to a simultaneous block-diagonal sum
\begin{equation}\label{eq:K-decomp}
(A,B)\sim
\bigoplus_{i=1}^{n-\rho}P_{a_i}
\oplus
\bigoplus_{j=1}^{m-\rho}Q_{b_j}
\oplus R,
\end{equation}
where $R\in\Mat_s(\D)^2$ is regular, and
\[
\rho=\sum_{i=1}^{n-\rho}a_i+\sum_{j=1}^{m-\rho}b_j+s.
\]
The multisets $\{a_i\}$ and $\{b_j\}$, and the equivalence class of
$R$, are uniquely determined by $(A,B)$.
\end{theorem}
\begin{proof}
For $\D=F$, this is the usual Kronecker theorem
(see \cite[Chapter~XII, \S\S3--5]{Gantmacher}). We reduce the division-algebra
case to that theorem, retaining the $\D$-module structure.
Regard the pencil as a Kronecker representation $M$ over $F$,
and put $B=\End_{\mathrm{Kr},F}(M)$ and $J=\operatorname{rad}B$.
Right multiplication embeds $\D^{\mathrm{op}}$ into $B$.
The isotypic components of the $F$-Kronecker decomposition give
orthogonal central idempotents in $B/J$.

Since $\D$ is separable over $F$, the map
\[
 C_B(\D^{\mathrm{op}})\longrightarrow
 C_{B/J}(\D^{\mathrm{op}})
\]
is surjective. Indeed, a separability idempotent gives an averaging
projection onto the centralizer, preserving $J$ and inducing the
corresponding projection modulo $J$. Its kernel is nilpotent.
Consequently the indicated central idempotents lift to a complete
orthogonal family in $C_B(\D^{\mathrm{op}})$. Their images split
$M$ into $\D$-stable summands whose underlying $F$-representations
are its original isotypic components.

For a singular indecomposable $P_a(F)$ one has
$\End_{\mathrm{Kr},F}(P_a(F))=F$. A $\D$-module structure on
$P_a(F)^{\oplus u}$ therefore amounts to a right $\D$-module
structure on the multiplicity space $F^u$. This space is free of
some rank $v$, with $u=d^2v$, and the resulting summand is
$P_a(\D)^{\oplus v}$. The same argument applies to $Q_b$.
The remaining summand is regular as an $F$-pencil, hence also
as a $\D$-pencil; its two $\D$-dimensions are equal.
This proves existence. The uniqueness of the underlying
$F$-decomposition gives the multiplicities of the singular blocks,
and Krull--Schmidt cancellation gives the regular summand.
Finally, $P_a$ has normal rank $a$, $Q_b$ has normal rank $b$,
and a regular square summand of size $s$ has rank $s$.
Comparing the two vertex dimensions gives the stated counts.
\end{proof}

For a fixed pair, write the three terms in \eqref{eq:K-decomp} as
$P\oplus Q\oplus R$, and put $\sigma=\sum_i a_i$.

\subsection{Morphisms between the blocks}
For matrix pairs $X=(A,B)\in\Mat_{m\times n}(\D)^2$ and
$X'=(A',B')\in\Mat_{m'\times n'}(\D)^2$, define
\[
\Hom_{\mathrm{Kr}}(X,X')
=
\left\{(S,T)\in
\Hom_{\D}(\D^n,\D^{n'})\times\Hom_{\D}(\D^m,\D^{m'})
:
A'S=TA,\ B'S=TB
\right\}.
\]
For a fixed pair $(A,B)$, let $S_{A,B}\subset L_{m,n}$ be its stabilizer.
Differentiating $gAh^{-1}=A$ and $gBh^{-1}=B$ at the identity gives
\[
\mathfrak s_{A,B}
=\left\{(X,Y)\in\Mat_m(\D)\oplus\Mat_n(\D):
XA=AY,\ XB=BY\right\}.
\]
Thus, via $(X,Y)\mapsto(Y,X)$,
\[
\mathfrak s_{A,B}\simeq
\Hom_{\mathrm{Kr}}\bigl((A,B),(A,B)\bigr).
\]
If $(A,B)=P\oplus(Q\oplus R)$, then
\[
\mathfrak s_{A,B}\simeq
\End_{\mathrm{Kr}}(P)
\oplus\End_{\mathrm{Kr}}(Q\oplus R)
\oplus\Hom_{\mathrm{Kr}}(Q\oplus R,P)
\oplus\Hom_{\mathrm{Kr}}(P,Q\oplus R),
\]
where $\End_{\mathrm{Kr}}(X)=\Hom_{\mathrm{Kr}}(X,X)$.
We record only the Hom spaces involving the $P$-blocks, since these are
the ones that contribute nontrivially to the adjoint determinant below.

\begin{proposition}\label{prop:K-Hom}
For $a,b\geq0$, let $P_a,Q_b$ be the standard blocks above, and let
$R\in\Mat_s(\D)^2$ be regular. Then
\begin{equation}\label{eq:K-Hom}
\begin{gathered}
\Hom_{\mathrm{Kr}}(P_a,Q_b)=0,
\qquad
\dim_F\Hom_{\mathrm{Kr}}(Q_b,P_a)=d^2(a+b),\\
\Hom_{\mathrm{Kr}}(P_a,R)=0,
\qquad
\dim_F\Hom_{\mathrm{Kr}}(R,P_a)=d^2s.
\end{gathered}
\end{equation}
\end{proposition}
\begin{proof}
The two assertions involving only $P_a$ and $Q_b$ follow directly from
their explicit block matrices.
Applying the same transformation
\[
 (A,B)\longmapsto(rA+tB,uA+vB),
 \qquad r,t,u,v\in F,\quad rv-tu\ne0,
\]
to both pairs preserves the Hom dimensions, as do changes of bases.
For the regular pair $R=(A,B)$, choose $r,t$ so that $rA+tB$ is
invertible. The transformed $P_a$ is equivalent to $P_a$, and left
multiplication by $(rA+tB)^{-1}$ puts the transformed $R$ in the form
$(\mathbf I_s,D)$, where $D=(rA+tB)^{-1}(uA+vB)$.
Thus we may use these standard forms in the remaining calculations.

Assume first that $a\geq1$. If
$(S,T)\in\Hom_{\mathrm{Kr}}(R,P_a)$, then
\[
S\in\Mat_{(a+1)\times s}(\D),\qquad
T\in\Mat_{a\times s}(\D),
\]
and the intertwining equations are
\[
A_aS=T,\qquad B_aS=TD.
\]
Writing the rows of $S$ as $z_0,\ldots,z_a\in\Mat_{1\times s}(\D)$ gives
\[
T=\begin{bmatrix}z_0\\ \vdots\\ z_{a-1}\end{bmatrix},
\qquad
z_{j+1}=z_jD\quad(0\leq j<a).
\]
Hence
\[
S=\begin{bmatrix}z_0\\ z_0D\\ \vdots\\ z_0D^a\end{bmatrix},
\qquad
T=\begin{bmatrix}z_0\\ z_0D\\ \vdots\\ z_0D^{a-1}\end{bmatrix},
\]
so $(S,T)$ is uniquely determined by $z_0\in\Mat_{1\times s}(\D)$ and therefore
$\dim_F\Hom_{\mathrm{Kr}}(R,P_a)=d^2s$.

Conversely, if $(S,T)\in\Hom_{\mathrm{Kr}}(P_a,R)$, then
\[
S\in\Mat_{s\times(a+1)}(\D),\qquad
T\in\Mat_{s\times a}(\D),
\]
and
\[
S=TA_a=\begin{bmatrix}T&\mathbf 0_{s\times1}\end{bmatrix},
\qquad
DS=TB_a=\begin{bmatrix}\mathbf 0_{s\times1}&T\end{bmatrix}.
\]
Writing $T=[t_1,\ldots,t_a]$ gives
\[
Dt_1=0,\qquad Dt_j=t_{j-1}\ (2\leq j\leq a),\qquad t_a=0.
\]
Thus $t_a=t_{a-1}=\cdots=t_1=0$, so $T=S=0$ and
$\Hom_{\mathrm{Kr}}(P_a,R)=0$. The case $a=0$ is immediate from the
defining equations.
\end{proof}

\subsection{Stabilizer characters}
Replace $(A,B)$ by the representative in \eqref{eq:K-decomp}, so that
\[
(A,B)=P\oplus Q\oplus R.
\]
For $t\in F^\times$, set
\[
g_t=\diag\!\left(t\mathbf I_{\sigma},\mathbf I_{m-\sigma}\right),
\qquad
h_t=\diag\!\left(t\mathbf I_{\sigma+n-\rho},
                  \mathbf I_{\rho-\sigma}\right).
\]
Thus on each $P_{a_i}$ block, $g_t$ acts as $t\mathbf I_{a_i}$ and
$h_t$ as $t\mathbf I_{a_i+1}$, while both act as the identity on all
$Q$-blocks and on $R$. A direct computation gives
$(g_t,h_t)\in S_{A,B}$.

On the four summands of $\mathfrak s_{A,B}$ displayed above,
conjugation by $(g_t,h_t)$ acts by multiplication by
$1,1,t,t^{-1}$, respectively. By \eqref{eq:K-Hom},
$\Hom_{\mathrm{Kr}}(P,Q\oplus R)=0$, while
\[
\begin{aligned}
\dim_F\Hom_{\mathrm{Kr}}(Q\oplus R,P)
 &=d^2\left(\sum_{i,j}(a_i+b_j)+(n-\rho)s\right)\\
 &=d^2\bigl((m-\rho)\sigma+(n-\rho)(\rho-\sigma)\bigr).
\end{aligned}
\]
Therefore
\begin{equation}\label{eq:K-delta}
\delta_{S_{A,B}}(g_t,h_t)
 =\nu(t)^{d^2((m-\rho)\sigma+(n-\rho)(\rho-\sigma))}.
\end{equation}

Now let us focus on the square case $m=n=r$.
\begin{lemma}\label{lem:square-slope}
Let $r\geq1$ and let $\kappa$ be a smooth character of $F^\times$. Set
\[
\chi_\kappa(g,h)=\kappa(\Nrd g/\Nrd h)\quad\mbox{ for }(g,h)\in L_{r,r}.
\]
Let $X\subset\cM_{r,r}$ be an $L_{r,r}$-invariant locally closed subset
consisting of singular pairs. If $\cD(X)^{L_{r,r},\chi_\kappa}\neq0$,
then $|\kappa|=\nu^{dj}$
for some integer $0\leq j\leq r-1$.
\end{lemma}
\begin{proof}
By Proposition~\ref{prop:orbitwise-localization}, there exists
$(A,B)\in X$ such that
\[
\chi_\kappa|_{S_{A,B}}=\delta_{S_{A,B}}^{-1}.
\]
Put $\rho=\operatorname{nrank}(A,B)<r$ and $k=r-\rho$. Then
\[
\frac{\Nrd g_t}{\Nrd h_t}=t^{-dk},
\qquad
\delta_{S_{A,B}}(g_t,h_t)=\nu(t)^{d^2k\rho}.
\]
Hence $\kappa(t)^{-dk}=\nu(t)^{-d^2k\rho}$, and therefore
$|\kappa(t)|=\nu(t)^{d\rho}$ for all $t\in F^\times$.
Since $0\leq\rho\leq r-1$, the result follows.
\end{proof}

For the general case, we have
\begin{lemma}\label{lem:rect-obstruction}
Suppose $p>q\geq1$, and set
\begin{equation}\label{eq:zeta}
\zeta_{p,q}(g,h)=\lambda(g)^q\lambda(h)^{-p}\mbox{ for }(g,h)\in L_{p,q}.
\end{equation}
Let $X\subset\cM_{p,q}$ be an $L_{p,q}$-invariant locally closed subset
consisting of pairs of normal column rank less than $q$. Then
\[
\cD(X)^{L_{p,q},\zeta_{p,q}}=0.
\]
\end{lemma}
\begin{proof}
Put $\Delta=p-q>0$. For $(A,B)\in X$, let
$\rho=\operatorname{nrank}(A,B)<q$.
For the element $(g_t,h_t)$ above,
\[
\zeta_{p,q}(g_t,h_t)=\nu(t)^{d^2(-\Delta\sigma-p(q-\rho))},
\qquad
\delta_{S_{A,B}}(g_t,h_t)^{-1}
 =\nu(t)^{-d^2((p-\rho)\sigma+(q-\rho)(\rho-\sigma))}.
\]
After division by $d^2$, the difference of the exponents is
\[
[-(p-\rho)\sigma-(q-\rho)(\rho-\sigma)]
 -[-\Delta\sigma-p(q-\rho)]
 =(q-\rho)(p-\rho)>0.
\]
Therefore
\[
\zeta_{p,q}|_{S_{A,B}}\neq\delta_{S_{A,B}}^{-1}
\]
on every orbit in $X$. Proposition~\ref{prop:orbitwise-localization}
then gives the vanishing.
\end{proof}

\section{Adjoint symmetry for pairs of square matrices}\label{sec:square}
In this section, we will focus on the square case, which is the key input to the proof of Theorem \ref{thm:group}.
\begin{theorem}\label{thm:square}
For every $r\geq1$ and every smooth character $\kappa$, every distribution
in $\cD(\cM_{r,r})^{L_{r,r},\chi_\kappa}$ is fixed by
$(A,B)\mapsto(A^{\dagger},B^{\dagger})$.
\end{theorem}
\begin{proof}
Let $\tau(A,B)=(A^{\dagger},B^{\dagger})$. For $(g,h)\in L_{r,r}$ and
$(A,B)\in\cM_{r,r}$,
\[
\tau\bigl((g,h)\cdot(A,B)\bigr)
=((h^{\dagger})^{-1},(g^{\dagger})^{-1})\cdot\tau(A,B).
\]
Moreover, $\chi_\kappa((h^{\dagger})^{-1},(g^{\dagger})^{-1})=\chi_\kappa(g,h)$. Since
$\tau_*$ stabilizes $\cD(\cM_{r,r})^{L_{r,r},\chi_\kappa}$ and
$\tau_*^2=1$, it remains to show that the $(-1)$-eigenspace is zero.
Let $T\in\cD(\cM_{r,r})^{L_{r,r},\chi_\kappa}$ satisfy $\tau_*T=-T$. We need to show that $T=0$.

Let us consider the regular locus first.
Set
\[
\Omega=\{(A,B)\in\cM_{r,r}:A\in\GL_r(\D)\}.
\]
Writing $C=A^{-1}B$ identifies
$\Omega\simeq\GL_r(\D)\times\Mat_r(\D)$, and in the coordinates $(A,C)$
the action is
\[
(g,h)\cdot(A,C)=(gAh^{-1},hCh^{-1}).
\]
Restricting to the first $\GL_r(\D)$-factor, $\Omega$ is the associated
space $L\times^S W$ with
$L=\GL_r(\D)$, $S=\{1\}$, and $W=\Mat_r(\D)$.
Lemma~\ref{lem:frob} therefore gives
\[
\cD(\Omega)^{L,\kappa\circ\Nrd}\simeq\cD(W).
\]
If $U\in\cD(W)$ corresponds to $T|_\Omega$, then
\[
T|_\Omega=\kappa(\Nrd A)\,d^\times A\otimes U.
\]
The second $\GL_r(\D)$-factor then forces $U$ to be
conjugation-invariant.
In these coordinates, $\tau$ sends
\[
(A,C)\longmapsto(A^{\dagger},(A^{\dagger})^{-1}C^{\dagger}A^{\dagger}).
\]
By Theorem~\ref{thm:conjugation}, $U$ is invariant under $g\mapsto g^\dagger$,
as is the measure $\kappa(\Nrd A)\,d^\times A$.
Hence every equivariant distribution is adjoint-invariant on
$\{A\in\GL_r(\D)\}$. The $\GL_2(F)$-action
\[
(A,B)\longmapsto(aA+bB,cA+eB),
\qquad
\begin{pmatrix}a&b\\c&e\end{pmatrix}\in\GL_2(F),
\]
commutes with the $L_{r,r}$-action and with the action of $\tau$. Thus the same
argument shows that $T$ vanishes wherever $aA+bB$ is invertible, for
every $(a,b)\neq(0,0)$. Therefore
\[
\supp T\subseteq
\{(A,B):\Nrd(sA+tB)\equiv0\},
\]
the locus of singular pairs.

Now we consider the Fourier transform.
Let $\cF$ be the Fourier transform on $\cM_{r,r}$ with respect to
\[
((A,B),(C,D))\longmapsto\Trd(AC)+\Trd(BD).
\]
It commutes with componentwise action of $\tau$. Relative to this pairing,
the action $(A,B)\mapsto(gAh^{-1},gBh^{-1})$ is dual to
$(C,D)\mapsto(hCg^{-1},hDg^{-1})$.
The former has absolute determinant $\bigl(\lambda(g)/\lambda(h)\bigr)^{2r}$.
Applying \eqref{eq:Fourier-equivariance} and then interchanging $g,h$
to express the result in terms of the original action gives
\begin{equation}\label{eq:Fourier-character}
\cF T\in\cD(\cM_{r,r})^{L_{r,r},\chi_{\kappa^*}},
\qquad \kappa^*=\nu^{2dr}\kappa^{-1}.
\end{equation}
Since $\cF$ commutes with $\tau_*$, we have
$\tau_*\cF T=-\cF T$. Hence the argument before on the regular-locus shows that
$\supp\cF T$ is contained in the singular locus.

Now we can use the contradiction argument to finish the proof.
If $T\neq0$, then $\cF T\neq0$. Applying
Lemma~\ref{lem:square-slope} to $T$ and $\cF T$ gives
\[
|\kappa|=\nu^{dj},
\qquad
|\kappa^*|=\nu^{dj'}
\]
for some $0\leq j,j'\leq r-1$. Since
$\kappa\kappa^*=\nu^{2dr}$ (see \eqref{eq:Fourier-character}), this would give
\[
\nu^{2dr}=|\kappa\kappa^*|=\nu^{d(j+j')},
\]
which is impossible because $j+j'\leq2r-2$. Hence $T=0$.
\end{proof}

\begin{corollary}\label{cor:balanced-recip}
Let $\GL_r(\D)\times\GL_r(\D)$ act on
$\Mat_r(\D)\oplus\Mat_r(\D)$ by
\[
(g,h)\cdot(X,Y)=(gXh^{-1},hYg^{-1}).
\]
For every smooth character $\gamma$, every distribution transforming by
$\gamma(\Nrd g/\Nrd h)$ is fixed by $(X,Y)\mapsto(X^{\dagger},Y^{\dagger})$.
\end{corollary}
\begin{proof}
Let $\cF_2$ denote Fourier transform in the second variable, with respect
to the pairing $(Y,Z)\mapsto\Trd(YZ)$:
\[
(\cF_2f)(X,Y)
 =\int_{\Mat_r(\D)}f(X,Z)\psi(\Trd(YZ))\,dZ.
\]
The action $Y\mapsto hYg^{-1}$ has absolute determinant
$\bigl(\lambda(h)/\lambda(g)\bigr)^r$ and is dual to $Z\mapsto gZh^{-1}$.
Thus \eqref{eq:Fourier-equivariance} gives
\[
\cF_2T\in
\cD(\cM_{r,r})^{L_{r,r},\chi_\kappa},
\qquad
\kappa=\gamma\nu^{dr},
\]
for the action $(g,h)\cdot(X,Z)=(gXh^{-1},gZh^{-1})$, where
\[
\chi_\kappa(g,h)
=\gamma(\Nrd g/\Nrd h)\,
 \bigl(\lambda(g)/\lambda(h)\bigr)^r.
\]
Moreover, $\Trd(Y^{\dagger}Z)=\Trd(YZ^{\dagger})$, so $\cF_2$ commutes with
$(X,Y)\mapsto(X^{\dagger},Y^{\dagger})$. Hence Theorem~\ref{thm:square} gives
\[
\tau_*\cF_2T=\cF_2T.
\]
Then the Fourier inversion gives $\tau_*T=T$.
\end{proof}
\begin{remark}
	When $d=1$, Corollary \ref{cor:balanced-recip} generalized \cite[Theorem C(ii)]{CS} when $F$ is nonarchimedean and $\gamma=\mathbf{1}$, which is the key in this paper for us to remove the good character condition in \cite{CS}. 
\end{remark}

\section{Rectangular symmetry}\label{sec:padding}
In this section, we will deal with the general case with $\gamma=\mathbf{1}$. First we have the following useful observation.
\begin{proposition}\label{prop:image-rank}
Let $p>q\geq1$. Every $\zeta_{p,q}$-equivariant distribution $T$ on
$\cM_{p,q}$ satisfies
\begin{equation}\label{eq:image-rank}
\supp T\subseteq
\{(A,B):\dim(\im A+\im B)\leq q\}.
\end{equation}
\end{proposition}
\begin{proof}
We shall reduce the problem to the square pairs.
Put $\Delta=p-q$ and consider the incidence variety
\[
\mathcal K=\{(A',B',K)\in\cM_{p,p}\times\Gr_\Delta(\D^p):A'K=B'K=0\}.
\]
Since the Grassmannian $\Gr_\Delta$ is projective and $\mathcal K$ is closed in
$\cM_{p,p}\times\Gr_\Delta$, the projection
$\pi:\mathcal K\to\cM_{p,p}$ is proper. The second $\GL_p(\D)$-factor acts transitively on $\Gr_\Delta(\D^p)$
by $h\cdot K=hK$. Fix the decomposition $\D^p=\D^q\oplus \D^\Delta$ and let
$K_0:=\{0\}\oplus \D^\Delta\in\Gr_\Delta(\D^p)$. The stabilizer of $K_0$ in
$\GL_p(\D)$ is
\[
Q=\left\{\begin{pmatrix}h&0\\ u&k\end{pmatrix}:
h\in\GL_q(\D),\ k\in\GL_\Delta(\D)\right\}.
\]
Every point in the fiber over $K_0$ is uniquely of the form
$([A,0],[B,0],K_0)$ with $(A,B)\in\cM_{p,q}$.

Put
\[
L=L_{p,p},\qquad W=\cM_{p,q},\qquad
S=\GL_p(\D)\times Q\subset L.
\]
Then $\mathcal K\simeq L\times^S W$. For
$q_0=\left(\begin{smallmatrix}h&0\\ u&k\end{smallmatrix}\right)\in Q$,
the induced action on $W$ is
\[
(g,q_0)\cdot(A,B)=(gAh^{-1},gBh^{-1}).
\]
Since $\delta_S=\delta_Q=\lambda(k)^q\lambda(h)^{-\Delta}$,
\[
(\chi_{\nu^{dq}}|_S\delta_S)(g,q_0)
=\lambda(g)^q\lambda(h)^{-p}
=\zeta_{p,q}(g,h).
\]
Hence Lemma~\ref{lem:frob} gives
\[
\cD(\mathcal K)^{L,\chi_{\nu^{dq}}}
\simeq
\cD(W)^{S,\chi_{\nu^{dq}}|_S\delta_S}
=
\cD(\cM_{p,q})^{L_{p,q},\zeta_{p,q}}.
\]
Thus $T$ corresponds to a distribution $\widetilde T$ on $\mathcal K$
of character $\chi_{\nu^{dq}}$. The proper pushforward gives
$T'=\pi_*\widetilde T$ on $\cM_{p,p}$ with that character.

By definition, $\supp T'\subseteq\pi(\supp\widetilde T)\subseteq\pi(\mathcal K)$. We get
\[
\supp T'\subseteq
\{(A',B'):\dim(\ker A'\cap\ker B')\ge \Delta\}.
\]
Theorem~\ref{thm:square} and the action of $\tau$ on $\cM_{p,p}$ give
\[
\supp T'\subseteq
\{(A',B'):\dim(\ker A'^{\dagger}\cap\ker B'^{\dagger})\ge \Delta\}.
\]
Equivalently,
\[
\supp T'\subseteq
\{(A',B'):\dim(\im A'+\im B')\le q\}.
\]

In order to transfer the bound from $T'$ to $T$, set
\[
\begin{split}
Z_\Delta&=\{(A',B'): \dim(\ker A'\cap\ker B')\geq \Delta\},\\
Z_\Delta^\circ&=\{(A',B'): \dim(\ker A'\cap\ker B')=\Delta\},\\
\mathcal K^\circ&=\pi^{-1}(Z_\Delta^\circ)\mbox{  and}\quad
\cM_{p,q}^\circ=\{(A,B):\ker A\cap\ker B=0\}.
\end{split}
\]
The set $Z_\Delta$ is closed and $Z_\Delta^\circ$ is open in it. The proper map
and its restriction fit into the Cartesian diagram
\[
\begin{CD}
\mathcal K^\circ @>>> \mathcal K\\
@V{\pi^\circ}_{\simeq}VV @VV{\pi}V\\
Z_\Delta^\circ @>>> Z_\Delta .
\end{CD}
\]
Above $(A',B')\in Z_\Delta^\circ$, the unique incidence subspace is
$K=\ker A'\cap\ker B'$. Thus $\pi^\circ$ is an isomorphism.
Under $\mathcal K\simeq L\times^S W$, one has
$\mathcal K^\circ\simeq L\times^S\cM_{p,q}^\circ$. By the naturality of
Frobenius descent under restriction,
$\widetilde T|_{\mathcal K^\circ}$ corresponds to
$T|_{\cM_{p,q}^\circ}$.

Since $\supp T'\subset Z_\Delta$, we regard $T'$ as a distribution on $Z_\Delta$.
Apply \eqref{eq:proper-restriction} to the proper map
$\mathcal K\to Z_\Delta$. It gives the identity
\[
 T'|_{Z_\Delta^\circ}=\pi^\circ_*
   (\widetilde T|_{\mathcal K^\circ}).
\]
On the open set
\[
U=\{(A',B')\in Z_\Delta^\circ:\dim(\im A'+\im B')>q\},
\]
the left side vanishes, hence
\[
(\widetilde T|_{\mathcal K^\circ})|_{(\pi^\circ)^{-1}(U)}=0.
\]
The
Frobenius descent therefore gives
\[
T|_{\{(A,B)\in\cM_{p,q}^\circ:\dim(\im A+\im B)>q\}}=0.
\]

It follows that
$\supp T\cap\{\dim(\im A+\im B)>q\}$
is a subset of $\{\ker A\cap\ker B\ne0\}$. On the latter locus,
the normal rank $\operatorname{nrank}(A,B)<q$, since any nonzero vector in
$\ker A\cap\ker B$ lies in the kernel of $A+tB$ over $\D(t)$.
Lemma~\ref{lem:rect-obstruction} therefore gives
\[
T|_{\{(A,B):\dim(\im A+\im B)>q\}}=0,
\]
proving \eqref{eq:image-rank}.
\end{proof}


The rest of this section is devoted to proving the following theorem, which is another key input to the proof of Theorem \ref{thm:group}. 
\begin{theorem}\label{thm:rectangular}
For all $p,q\geq0$, let $\GL_p(\D)\times\GL_q(\D)$ act on
$\Mat_{p\times q}(\D)\oplus\Mat_{q\times p}(\D)$ by
\[
(g,h)\cdot(X,Y)=(gXh^{-1},hYg^{-1}).
\]
Then every invariant distribution is fixed by
\begin{equation}\label{eq:rect-adjoint}
\tau'(X,Y)=(Y^{\dagger},X^{\dagger}).
\end{equation}
\end{theorem}

We rule out adjoint anti-invariant distributions by partial Fourier
transform followed by the Radon parametrization. We first establish
this parametrization and its compatibility with Fourier transform;
the final auxiliary lemma extends vanishing from the standard chart
to the regular locus, then uses descent on the singular locus.

\subsection{A Radon parametrization}
Throughout the auxiliary constructions below, assume $p\ge q\ge1$.
Let $\mathcal G=\Gr_q^\ell(\D^{2q})$ be the left Grassmannian. For $U\in\mathcal G$, let
\[
W_U=\{M\in\Mat_{p\times2q}(\D):\text{each row of }M\text{ lies in }U\}
\]
and let $\cL\to\mathcal G$ be the
Haar line bundle whose fiber $\cL_U$ is the one-dimensional complex
space spanned by Haar measures on $W_U$. Define
\begin{equation}\label{eq:Radon}
\cR:\cS(\Mat_{p\times2q}(\D))
\longrightarrow\cS(\mathcal G,\cL^\vee).
\end{equation}
For $f\in\cS(\Mat_{p\times2q}(\D))$, $U\in\mathcal G$, and
$\mu\in\cL_U$, set
\[
(\cR f)(U)(\mu)=\int_{W_U}f\,d\mu.
\]

\begin{lemma}\label{lem:Radon}
The dual of \eqref{eq:Radon} identifies
$\cS(\mathcal G,\cL^\vee)^*$ with the distributions on
$\Mat_{p\times2q}(\D)$ supported in rank at most $q$ and satisfying
\[
S\bigl(f(g^{-1}\,\cdot)\bigr)=\lambda(g)^qS(f),\qquad g\in\GL_p(\D).
\]
\end{lemma}
\begin{proof}
Write $Z_{\leq q}$ for the closed rank-at-most-$q$ locus and
$Z_q$ for its rank-$q$ stratum. We first identify the equivariant
distributions on $Z_q$, then show that no such distributions are
supported on the lower-rank boundary.

For $M\in Z_q$, its rows span a $q$-dimensional left $\D$-subspace
$U\subset \D^{2q}$, hence a point of $\mathcal G$. On a chart
$\Omega\subset\mathcal G$, choose for each $U\in\Omega$ a matrix
$R(U)\in\Mat_{q\times2q}(\D)$ whose rows form a left $\D$-basis of $U$. Then every
$M\in Z_q$ whose rows span $U$ is uniquely of the form
\[
M=XR(U),\qquad
X\in\Mat_{p\times q}(\D),\quad \rank X=q.
\]
Thus, over $\Omega$, the rank-$q$ stratum is parametrized by $U\in\Omega$
and full-rank matrices $X\in\Mat_{p\times q}(\D)$. The group $\GL_p(\D)$
acts on the latter by $X\mapsto gX$ and is transitive on the full-rank
locus.

Put
\[
X_0=\begin{pmatrix}I_q\\0\end{pmatrix},\qquad
S_q=\operatorname{Stab}_{\GL_p(\D)}(X_0)
=\left\{
\begin{pmatrix}I_q&b\\0&a\end{pmatrix}:
 a\in\GL_{p-q}(\D)
\right\}.
\]
The part of $Z_q$ corresponding to $\Omega$ is therefore
\[
\GL_p(\D)\times^{S_q}\Omega,
\qquad [g,U]\longmapsto gX_0R(U),
\]
where $S_q$ acts trivially on $\Omega$. For
$s=\begin{pmatrix}I_q&b\\0&a\end{pmatrix}$ one has
$\delta_{S_q}(s)=\lambda(a)^{-q}$ and
$\lambda(s)^q=\lambda(a)^q$. Hence Lemma~\ref{lem:frob} gives, on
each such chart, an isomorphism between the $\lambda^q$-equivariant
distributions on the corresponding part of $Z_q$ and distributions on
$\Omega$.

The chosen row basis $R(U)$ identifies additive Haar measure $dX$ on
$\Mat_{p\times q}(\D)$ with a Haar measure $\mu_R$ on $W_U$. If $R$ is
replaced by $aR$, then
\begin{equation}\label{eq:Haar-transition}
\mu_{aR}(f)=\int f(XaR)\,dX
 =\lambda(a)^{-p}\mu_R(f).
\end{equation}
These are the transition functions of the Haar line bundle $\cL$.
Hence the local Frobenius identifications are compatible on overlaps.
Since, under $M=XR(U)$, the local map is integration in the $X$-variable,
they glue to
\[
\cR_q^*:\cS(\mathcal G,\cL^\vee)^*
\xrightarrow{\sim}
\cD(Z_q)^{\GL_p(\D),\lambda^q},
\]
where $\cR_q$ is the restriction of $\cR$ to $Z_q$. This proves the
assertion on the rank-$q$ stratum.

Now we focus on the lower-rank boundary.
For $0\le k<q$, let
\[
Z_k=\{M\in\Mat_{p\times2q}(\D):\rank M=k\}.
\]
By the same argument, with $k$ in place of $q$ in the parametrization
of the stratum, the Frobenius descent would require
\[
\lambda(a)^q=\lambda(a)^k
\qquad(a\in\GL_{p-k}(\D)).
\]
Since $k<q$, this is impossible. Hence
\[
\cD(Z_k)^{\GL_p(\D),\lambda^q}=0
\qquad(0\le k<q),
\]
and therefore
\[
\cD_{Z_{\le q-1}}(\Mat_{p\times2q}(\D))^{\GL_p(\D),\lambda^q}=0.
\]

Thus we have the following conclusion:
for $D\in\cS(\mathcal G,\cL^\vee)^*$, defining
$(\cR^*D)(f)=D(\cR f)$, then $\cR^*D$ is supported in $Z_{\le q}$ and
has left character $\lambda^q$. Conversely, if
\[
S\in
\cD_{Z_{\le q}}(\Mat_{p\times2q}(\D))^{\GL_p(\D),\lambda^q},
\]
the rank-$q$ bijection gives a unique
$D\in\cS(\mathcal G,\cL^\vee)^*$ such that
$S|_{Z_q}=\cR_q^*D$. Hence $S-\cR^*D$ is supported in
$Z_{\le q-1}$, and the preceding vanishing gives $S=\cR^*D$.
Injectivity follows by restricting to $Z_q$ and using the injectivity
of $\cR_q^*$. Thus $\cR^*$ is a bijection.
\end{proof}

\subsection{Fourier duality on the Haar lines}
Give the left row space $\D^{2q}=\D^q\oplus\D^q$ the
nondegenerate skew-Hermitian form
\[
 b((x,z),(x',z'))=x(z')^\dagger-z(x')^\dagger.
\]
For a left subspace $U$, put
$U^\perp=\{v:b(u,v)=0\text{ for all }u\in U\}$.
The induced alternating $F$-bilinear form on the matrix space is
\begin{equation}\label{eq:alternating}
 \mathcal B((X,Z),(X',Z'))
 =\Trd\bigl(X(Z')^\dagger-Z(X')^\dagger\bigr).
\end{equation}
Let $\mathscr F$ be the corresponding Fourier transform with the
fixed self-dual Haar measure. Then
\begin{equation}\label{eq:annihilator}
 W_U^{\perp_{\mathcal B}}=W_{U^\perp},\qquad
 \iota:\mathcal G\longrightarrow\mathcal G,\quad U\longmapsto U^\perp.
\end{equation}
Indeed, rows may be varied independently. Since $U$ is a left
$\D$-subspace, vanishing of $\Trd(b(u,v))$ for all $u\in U$
also gives $\Trd(a\,b(u,v))=0$ for every $a\in\D$.
Nondegeneracy of the reduced trace pairing gives $b(u,v)=0$.
The map $\iota$ is an algebraic involution of the left
$\D$-Grassmannian.

\begin{proposition}\label{prop:Radon-Fourier}
There is an isomorphism of smooth Haar-line bundles
\[
 J_U:\cL_U\longrightarrow\cL_{\iota(U)},\qquad
 J_U(\mu)=\mathscr F\mu,\quad \mu\in\cL_U,
\]
where $\mu$ is viewed as a distribution on the ambient matrix space
supported on $W_U$. Moreover,
\[
J_{\iota(U)}\circ J_U=\operatorname{id}_{\cL_U}.
\]
For $\Phi\in\cS(\mathcal G,\cL^\vee)$, define
\begin{equation}\label{eq:Haar-section-involution}
 (\mathcal I\Phi)(U)(\mu)=\Phi(\iota(U))(J_U\mu),
 \qquad U\in\mathcal G,\quad \mu\in\cL_U.
\end{equation}
Then $\mathcal I^2=1$, and the diagram
\[
\begin{CD}
\cS(\Mat_{p\times2q}(\D)) @>{\mathscr F}>>
       \cS(\Mat_{p\times2q}(\D))\\
@V{\cR}VV @VV{\cR}V\\
\cS(\mathcal G,\cL^\vee) @>{\mathcal I}>>
       \cS(\mathcal G,\cL^\vee)
\end{CD}
\]
commutes. In particular, if $S=\cR^*D$ as in
Lemma~\ref{lem:Radon}, then
\begin{equation}\label{eq:Radon-Fourier-dual}
 \mathscr F S=\cR^*(\mathcal I^*D),\qquad
 \mathcal I^*D=D\circ\mathcal I.
\end{equation}
Consequently $\mathscr F S=-S$ if and only if $\mathcal I^*D=-D$.
\end{proposition}
\begin{proof}
Equations~\eqref{eq:Fourier-subspace} and \eqref{eq:annihilator}
give the linear isomorphism $J_U$. In local choices of bases, the
quotient and dual measure normalizations are absolute values of
nonzero determinants, hence are locally constant. Thus these maps
form a smooth bundle isomorphism, compatible with changes of local
Haar measure. Since $\mathcal B$ is alternating, Fourier inversion gives
$J_{\iota(U)}\circ J_U=\operatorname{id}_{\cL_U}$.

For $U\in\mathcal G$ and $\mu\in\cL_U$,
\[
 (\cR\mathscr Ff)(U)(\mu)
 =\mu(\mathscr Ff)
 =(J_U\mu)(f)
 =(\mathcal I\cR f)(U)(\mu).
\]
Thus the diagram commutes. Applying $D$ gives
\eqref{eq:Radon-Fourier-dual}, and injectivity of $\cR^*$ gives
the last assertion.
\end{proof}

\subsection{Vanishing on the Grassmannian}
In this section, we will prove the vanishing result on the Grassmannian. 
\begin{lemma}\label{lem:Grassmannian-vanishing}
Suppose $p\ge q\ge1$, let
$D\in\cS(\mathcal G,\cL^\vee)^*$, and put $S=\cR^*D$.
If
\begin{equation}\label{eq:Grassmannian-conditions}
 S\bigl(f(\,\cdot\,h,\,\cdot\,h)\bigr)
 =\lambda(h)^{-p}S(f)
 \quad(h\in\GL_q(\D)),\qquad
 \mathcal I^*D=-D,
\end{equation}
then $D=0$.
\end{lemma}
\begin{proof}
	We will divide the proof into $3$ parts.
	
\emph{Step 1:} We first prove the vanishing result on the standard chart.
For $C\in\Mat_q(\D)$, let $U_C=\{(x,xC):x\in \D^q\}$.
The map $C\mapsto U_C$ identifies $\Mat_q(\D)$ with the open chart
$\Omega\subset\mathcal G$ on which the first projection
$U\to \D^q$ is an isomorphism. Using this projection and the fixed
Haar measure $dX$, we regard $D|_\Omega$ as a distribution
$D_\Omega$ on $\Mat_q(\D)$. In these coordinates, we write
$(\cR f)(C)$ for the scalar representative of the Radon section:
\[
 (\cR f)(C)=\int_{\Mat_{p\times q}(\D)}f(X,XC)\,dX.
\]
The skew-Hermitian orthogonal complement is $U_C^\perp=U_{C^\dagger}$.
Moreover,
\[
 \mathcal B\bigl((X,XC),(X',X'C^\dagger+V)\bigr)=\Trd(XV^\dagger).
\]
The change $(X',V)\mapsto(X',X'C^\dagger+V)$ has Jacobian one.
The Fourier inversion in $X,V$, with the fixed self-dual measures, gives
\[
 (\cR\mathscr Ff)(C)=(\cR f)(C^{\dagger}).
\]
Thus $\mathcal I$ acts in these coordinates by the adjoint $X\mapsto X^\dagger$, and
$\mathcal I^*D=-D$ makes $D_\Omega$ adjoint anti-invariant.
For $h\in\GL_q(\D)$, the change of variables $X\mapsto Xh$ gives
\[
 \cR\bigl(f(\,\cdot\,h,\,\cdot\,h)\bigr)(C)
 =\lambda(h)^{-p}(\cR f)(h^{-1}Ch).
\]
Since $S=\cR^*D$ and $\cR^*$ is injective, the factor
$\lambda(h)^{-p}$ cancels against the covariance factor of $S$.
Hence $D_\Omega$ is conjugation-invariant. Theorem~\ref{thm:conjugation}
makes it adjoint-invariant as well, i.e., it is invariant under $X\mapsto X^\dagger$, so $D_\Omega=0$.

\emph{Step 2:} Let us look at the situation on the regular locus.
For $U\in\mathcal G$, choose $A,B\in\Mat_q(\D)$ so that the rows
of $[A,B]$ form a basis of $U$. Call $U$ regular if $sA+tB$ is
invertible for some $(s,t)\in F^2\setminus\{(0,0)\}$. The condition is independent of the
chosen basis. Let
$\gamma=\left(\begin{smallmatrix}a&b\\c&e\end{smallmatrix}\right)
\in\SL_2(F)$ act by
\[
 \gamma\cdot(X,Z)=(aX+bZ,cX+eZ).
\]
Proposition~\ref{prop:Radon-Fourier} gives $\mathscr F S=-S$.
The action of $\gamma$ commutes with simultaneous right multiplication and
preserves $\mathcal B$ and its self-dual measure. By
\eqref{eq:Fourier-covariance}, it also commutes with $\mathscr F$.
Thus $\gamma_*S$
has the same right covariance as $S$ and remains Fourier
anti-invariant. Since $\gamma W_U=W_{\gamma U}$, change of variables gives
\[
 (\cR(f\circ\gamma))(U)(\mu)
 =\int_{W_U}f(\gamma M)\,d\mu(M)
 =(\cR f)(\gamma U)(\gamma_*\mu).
\]
Dualizing yields $\gamma_*S=\cR^*(\gamma_*D)$.
Proposition~\ref{prop:Radon-Fourier}
then shows that $\gamma_*D$ satisfies both hypotheses of the lemma.
Applying Step 1 to $\gamma_*D$ gives $D=0$
on $\gamma^{-1}\Omega$.
These open sets cover the regular locus: a nonzero row
$(s,t)$ can be completed to an element of $\SL_2(F)$, and the
first block of the transformed matrix is $sA+tB$.
Thus $D$ vanishes on the regular locus.

\emph{Step 3:} Now we can descent to the singular locus.
By Step 2, $D$ is supported on the singular locus.
To lift $D$ to a distribution on matrix pairs, set
\[
 \mathcal V=\{(A,B)\in\cM_{q,q}:\rank([A,B])=q\},\qquad
 \pi(A,B)=\operatorname{rowspan}[A,B].
\]
Two points of $\mathcal V$ have the same image under $\pi$ if and
only if they differ by a unique simultaneous left multiplication.
Thus the fiber product is
\[
 \GL_q(\D)\times\mathcal V\xrightarrow{\sim}
 \mathcal V\times_{\mathcal G}\mathcal V,\qquad
 (g,(A,B))\longmapsto\bigl((gA,gB),(A,B)\bigr).
\]
On a chart $\Omega\subset\mathcal G$, choose a matrix $R(U)$ whose
rows form a basis of $U$. The local model for descent is
\[
 \GL_q(\D)\times^{\{1\}}\Omega\xrightarrow{\sim}\pi^{-1}(\Omega),
 \qquad [g,U]\longmapsto gR(U).
\]
We apply Lemma~\ref{lem:frob} locally with character $\lambda^p$ on
$\GL_q(\D)$. The following integration formula makes the gluing
through the Haar line bundle explicit.
For $R=[A,B]\in\pi^{-1}(U)$, put
\[
 \mu_R(f)=\int_{\Mat_{p\times q}(\D)}f(XR)\,dX.
\]
As in \eqref{eq:Haar-transition},
$\mu_{gR}=\lambda(g)^{-p}\mu_R$.
Define
\[
 \mathcal A:\cS(\mathcal V)\longrightarrow
 \cS(\mathcal G,\cL^\vee)
\]
by
\[
 (\mathcal A\phi)(U)(\mu_R)
 =\int_{\GL_q(\D)}\phi(kR)\lambda(k)^p\,d^\times k.
\]
Replacing $R$ by $gR$ multiplies the integral by
$\lambda(g)^{-p}$, so the formula defines a functional on
$\cL_U$ independently of the chosen basis of $U$. Local trivializations of
$\pi$, together with a compactly supported function of weighted
integral one on $\GL_q(\D)$, show that $\mathcal A$ is surjective.
Thus the map $D\mapsto\widetilde D:=D\circ\mathcal A$ is injective.

Changing variables in the defining integrals for $\cR$ and
$\mathcal A$, and using the covariance of $S=\cR^*D$ and
injectivity of $\cR^*$, we obtain
\[
 \widetilde D\bigl(\phi(g^{-1}\,\cdot\,h)\bigr)
 =\bigl(\lambda(g)/\lambda(h)\bigr)^p\widetilde D(\phi)
 \qquad(g,h\in\GL_q(\D)).
\]
Thus $\widetilde D\in\cD(\mathcal V)^{L_{q,q},\chi_{\nu^{dp}}}$.
The singular locus is represented by pairs with the normal rank
$\operatorname{nrank}(A,B)<q$. Since
$\supp\widetilde D\subseteq\pi^{-1}(\supp D)$, the lift is supported
on the singular square pairs in $\mathcal V$. If $\widetilde D$
were nonzero, Lemma~\ref{lem:square-slope}, applied to this closed
subset of $\mathcal V$, would give $\nu^{dp}=\nu^{dj}$ for some
$0\le j\le q-1$. This contradicts $p\ge q$.
Hence $\widetilde D=0$, and surjectivity of $\mathcal A$ gives $D=0$. This finishes the proof.
\end{proof}
Finally we can give the proof of Theorem \ref{thm:rectangular}.
\begin{proof}[Proof of Theorem~\ref{thm:rectangular}]
 By symmetry assume $p\ge q\ge1$.
Since $\tau'_*$ is of order $2$, it suffices to show that an invariant distribution
$T$ with $\tau'_*T=-T$ is zero. We use the partial Fourier transform
and the Radon parametrization to construct a distribution satisfying
Lemma~\ref{lem:Grassmannian-vanishing}.

For $f\in\cS(\cM_{p,q})$, define
\[
 (Pf)(X,Y)=\int_{\Mat_{p\times q}(\D)}
 f(X,Z)\psi(\Trd(YZ))\,dZ,\qquad S(f)=T(Pf).
\]
For the trace pairing, $Y\mapsto hYg^{-1}$ is dual to
$Z\mapsto gZh^{-1}$, and its absolute determinant is
$\lambda(h)^p\lambda(g)^{-q}$.
Thus \eqref{eq:Fourier-covariance}, with the action
$\rho(g,h)f(X,Z)=f(g^{-1}Xh,g^{-1}Zh)$, gives
\[
 P(\rho(g,h)f)(X,Y)
 =\zeta_{p,q}(g,h)(Pf)(g^{-1}Xh,h^{-1}Yg),
\]
where $\zeta_{p,q}(g,h)=\lambda(g)^q\lambda(h)^{-p}$.
Thus $S$ is $\zeta_{p,q}$-equivariant. Moreover, \eqref{eq:reduced-identities} and the Fourier inversion give
$P^{-1}(\tau')^*P=\mathscr F$, so $\mathscr F S=-S$.

Thanks to Proposition~\ref{prop:image-rank}, we have 
\[
 \supp S\subseteq\{(X,Z):\rank([X,Z])\le q\}
\]
when $p>q$; this inclusion is automatic when $p=q$.
Lemma~\ref{lem:Radon} therefore yields a unique
$D\in\cS(\mathcal G,\cL^\vee)^*$ with $S=\cR^*D$.
The right covariance of $S$ gives the first hypothesis of
Lemma~\ref{lem:Grassmannian-vanishing}, while
Proposition~\ref{prop:Radon-Fourier} gives $\mathcal I^*D=-D$.
Then Lemma  \ref{lem:Grassmannian-vanishing}
implies $D=0$, hence $S=0$. Since $P$ is invertible, $T=0$. This finishes the proof.
\end{proof}
\begin{remark}
    When $d=1$, it is expected that every $\GL_p(F)\times\GL_{q-1}(F)$-invariant distribution on $\Mat_{p\times q}\oplus\Mat_{q\times p}(F)$ is also fixed by $\tau':(X,Y)\mapsto(Y^t,X^t)$, where $\GL_{q-1}$ is regarded a subgroup of the mirabolic subgroup of $\GL_q(F)$.
\end{remark}

\section{Closed double orbits and normal representations}\label{sec:group-descent}
The rest of this paper is devoted to the proof of Theorem \ref{thm:group}.
Regard $G$ and $H$ as algebraic groups over $F$. Interchanging the blocks
allows us to assume $p\geq q$; put $\Delta=p-q$. Let $L=H\times H$ act on $G$ by left and right multiplication,
and set
\[
\chi(h_l,h_r)=\eta(h_l)\eta(h_r)^{-1}.
\]
The adjoint $s(g)=g^{\dagger}$ normalizes this action by
\begin{equation}\label{eq:s-normalization}
s(h_lgh_r^{-1})=(h_r^{\dagger})^{-1}s(g)h_l^{\dagger}.
\end{equation}
Since $\eta((h^{\dagger})^{-1})=\eta(h)^{-1}$, $\chi$ is fixed by this automorphism
of $L$. Extend $\chi$ to $\widetilde L=L\rtimes\langle s\rangle$
by setting $\chi(s)=-1$, and retain the notation $\chi$.

We apply Theorem~\ref{thm:HC-descent} with
$K=\widetilde L$, $X=G$, and character $\chi$.
The normal space to the $\widetilde L$-orbit at $g_0$ is
$N_{g_0}=T_{g_0}G/T_{g_0}(L\cdot g_0)$,
with the induced action of $\widetilde L_{g_0}$.
It remains to prove
$\cD(N_{g_0})^{\widetilde L_{g_0},\chi}=0$
for every closed orbit.

\subsection{Rational closed-orbit representatives}
Let $V=V_+\oplus V_-$, with dimensions $p,q$, and put
$\varepsilon=\diag(I_p,-I_q)$. For semisimple $C\in\Mat_k(\D)$ with
$C(I-C)$ invertible, write
\begin{equation}\label{eq:gC}
g_C=\begin{pmatrix}I_k&C\\ I_k&I_k\end{pmatrix},\qquad
A_C=\{Z\in\Mat_k(\D):ZC=CZ\}.
\end{equation}
A zero-dimensional factor in all formulas is omitted.

\begin{lemma}\cite[Proposition 4.1]{JR} \label{lem:closed-orbits}
Every $L(F)$-orbit in $G(F)$ whose algebraic orbit is closed has a representative
\begin{equation}\label{eq:g0}
g_0=\left(
\begin{array}{ccc|ccc}
I_{a+\Delta}&0&0&0&0&0\\
0&0&0&0&I_b&0\\
0&0&I_k&0&0&C\\ \hline
0&0&0&I_a&0&0\\
0&I_b&0&0&0&0\\
0&0&I_k&0&0&I_k
\end{array}
\right),\qquad a+b+k=q.
\end{equation}
The block sizes are $(a+\Delta,b,k)$ in $V_+$ and $(a,b,k)$ in $V_-$;
the dividing lines separate $V_+$ from $V_-$.
We call the subspaces formed by block positions $(1,4)$, $(2,5)$,
and $(3,6)$ the \emph{identity summand}, \emph{exchange summand},
and \emph{generic summand}, respectively. They are $g_0$-invariant,
and the restrictions of $g_0$ to them are, respectively,
\[
 I_{2a+\Delta},\qquad
 \begin{pmatrix}0&I_b\\I_b&0\end{pmatrix},\qquad
 g_C=\begin{pmatrix}I_k&C\\I_k&I_k\end{pmatrix}.
\]
\end{lemma}
\begin{proof}
Put $\epsilon'=g\varepsilon g^{-1}$ and
$r=\epsilon'\varepsilon$. Closedness implies that $r$ is
semisimple by \cite[Proposition~7.2.1(i)]{AG}.
The relation $\varepsilon r\varepsilon=r^{-1}$ gives a
$\varepsilon$- and $\epsilon'$-stable decomposition of the right
$\D$-space $V$ into the $1$-eigenspace of $r$, its
$(-1)$-eigenspace, and the complementary summand $V_*$. This
is the rational primary decomposition, so all summands are
defined over $F$ and are right $\D$-subspaces.
On the first two summands one has, respectively,
$\epsilon'=\varepsilon$ and $\epsilon'=-\varepsilon$.

On $V_*$, decompose according to $\varepsilon$ and write
\[
 \epsilon'=\begin{pmatrix}A&B\\D_0&E\end{pmatrix}.
\]
Here $r+r^{-1}=\diag(2A,-2E)$ is semisimple and has no eigenvalues
$2$ or $-2$. Hence $I-A^2$ and $I-E^2$ are invertible.
From $(\epsilon')^2=I$ we get
$BD_0=I-A^2$ and $D_0B=I-E^2$, so $B$ and $D_0$ are
isomorphisms. In particular, the two $\varepsilon$-eigenspaces
in $V_*$ have the same $\D$-dimension $k$.
After a block-diagonal change of basis making $D_0=I_k$, the
same equations give
\[
 \epsilon'=\begin{pmatrix}A&I_k-A^2\\ I_k&-A\end{pmatrix},
 \qquad A\text{ semisimple},\quad A\pm I_k\text{ invertible}.
\]
Put $C=(A-I_k)(A+I_k)^{-1}$. Then $C$ is semisimple and
$C(I_k-C)$ is invertible. Conjugating the last matrix by
$\diag(I_k,A+I_k)$ gives
\[
 \begin{pmatrix}A&I_k-A\\A+I_k&-A\end{pmatrix}
 =g_C\diag(I_k,-I_k)g_C^{-1}.
\]
Thus this summand is represented by $g_C$.

Since $\epsilon'$ is conjugate to $\varepsilon$ and the generic
summand is balanced, the two $\varepsilon$-eigenspaces on the
$(-1)$-eigenspace of $r$ have the same dimension $b$. This gives
the exchange summand. The remaining dimensions are $a+\Delta,a$,
giving the identity summand, and $a+b+k=q$.
Finally, $g\mapsto g\varepsilon g^{-1}$ identifies $G/H$ with
the conjugacy class of $\varepsilon$. The preceding
$H$-conjugations therefore give the required $H\times H$
representative over $F$.
\end{proof}

Since $L$ has finite index in $\widetilde L$,
Lemma~\ref{lem:closed-orbits} implies that every closed
$\widetilde L$-orbit contains a point of the form \eqref{eq:g0}.

\subsection{The stabilizer, its character, and the full normal space}
For the representative $g_0$ in \eqref{eq:g0}, let
\[
 L_{g_0}=\{(h_l,h_r)\in L:h_lg_0h_r^{-1}=g_0\}
\]
be its stabilizer. Write $\fh=\Lie(H)$. Under
$T_{g_0}G\simeq\Mat_{p+q}(\D)$, the normal space is
\[
 N_{g_0}\simeq\Mat_{p+q}(\D)/(\fh g_0+g_0\fh).
\]
The stabilizer acts on this quotient by
\[
 (h_l,h_r)\cdot[Z]=[h_lZh_r^{-1}].
\]
We call $N_{g_0}$ with this action the normal representation at $g_0$.

\begin{lemma}\cite[Theorem 7.7.3]{AG}\label{lem:group-normal}
The stabilizer and its normal representation admit the identifications
\begin{align}
L_{g_0}&\simeq
 (\GL_{a+\Delta}(\D)\times\GL_a(\D))\times(\GL_b(\D)\times\GL_b(\D))\times A_C^\times,
 \label{eq:group-stabilizer}\\
N_{g_0}&\simeq
 \bigl(\Mat_{(a+\Delta)\times a}(\D)\oplus\Mat_{a\times(a+\Delta)}(\D)\bigr)
 \oplus\bigl(\Mat_b(\D)\oplus\Mat_b(\D)\bigr)\oplus A_C.
 \label{eq:group-normal}
\end{align}
Under these identifications, each of the first two stabilizer
factors acts on its corresponding normal summand by
\[
 (u,v)\cdot(X,Y)=(uXv^{-1},vYu^{-1}),
\]
while $A_C^\times$ acts on $A_C$ by $Z\mapsto cZc^{-1}$.
Each stabilizer factor acts trivially on the other summands.
Writing $\gamma=\alpha\beta^{-1}$ and denoting the middle stabilizer
factor by $(u_1,v_1)$, one has
$\chi|_{L_{g_0}}=\gamma(\Nrd u_1/\Nrd v_1)$.
In particular, the unbalanced normal factor has trivial character.
\end{lemma}
\begin{proof}
 We begin to compute the stabilizer and its character.
Write $h_l=\diag(P,Q)$ and $h_r=\diag(R,S)$, where
$P,R\in\GL_p(\D)$ and $Q,S\in\GL_q(\D)$.
Decompose $P=(P_{ij})$ and $R=(R_{ij})$ into $3\times3$ blocks
of sizes $(a+\Delta,b,k)$, and $Q=(Q_{ij})$ and $S=(S_{ij})$
into blocks of sizes $(a,b,k)$.
Using \eqref{eq:g0}, the four blocks of $h_lg_0=g_0h_r$
relative to $V_+\oplus V_-$ give
\begin{gather*}
\begin{pmatrix}
P_{11}&0&P_{13}\\ P_{21}&0&P_{23}\\ P_{31}&0&P_{33}
\end{pmatrix}
=
\begin{pmatrix}
R_{11}&R_{12}&R_{13}\\ 0&0&0\\ R_{31}&R_{32}&R_{33}
\end{pmatrix},\qquad
\begin{pmatrix}
0&P_{12}&P_{13}C\\ 0&P_{22}&P_{23}C\\ 0&P_{32}&P_{33}C
\end{pmatrix}
=
\begin{pmatrix}
0&0&0\\ S_{21}&S_{22}&S_{23}\\ CS_{31}&CS_{32}&CS_{33}
\end{pmatrix},\\[8pt]
\begin{pmatrix}
0&Q_{12}&Q_{13}\\ 0&Q_{22}&Q_{23}\\ 0&Q_{32}&Q_{33}
\end{pmatrix}
=
\begin{pmatrix}
0&0&0\\ R_{21}&R_{22}&R_{23}\\ R_{31}&R_{32}&R_{33}
\end{pmatrix},\qquad
\begin{pmatrix}
Q_{11}&0&Q_{13}\\ Q_{21}&0&Q_{23}\\ Q_{31}&0&Q_{33}
\end{pmatrix}
=
\begin{pmatrix}
S_{11}&S_{12}&S_{13}\\ 0&0&0\\ S_{31}&S_{32}&S_{33}
\end{pmatrix}.
\end{gather*}
Since $C$ is invertible, comparison of entries forces all
off-diagonal blocks of $P,Q,R,S$ to vanish. The diagonal entries satisfy
\[
\begin{gathered}
P_{11}=R_{11},\qquad Q_{11}=S_{11},\qquad
P_{22}=S_{22},\qquad Q_{22}=R_{22},\\
P_{33}=Q_{33}=R_{33}=S_{33},\qquad P_{33}C=CP_{33}.
\end{gathered}
\]
Consequently,
\begin{align*}
h_l&=\diag(u_0,u_1,c,v_0,v_1,c),\\
h_r&=\diag(u_0,v_1,c,v_0,u_1,c),
\end{align*}
where $u_0\in\GL_{a+\Delta}(\D)$, $v_0\in\GL_a(\D)$,
$u_1,v_1\in\GL_b(\D)$, and $c\in\GL_k(\D)$ satisfies $cC=Cc$,
that is, $c\in A_C^\times$.
Conversely, these matrices satisfy $h_lg_0=g_0h_r$.
Their reduced norms give
\[
 \chi(h_l,h_r)
 =\alpha\!\left(\frac{\Nrd u_1}{\Nrd v_1}\right)
  \beta\!\left(\frac{\Nrd v_1}{\Nrd u_1}\right)
 =\gamma\!\left(\frac{\Nrd u_1}{\Nrd v_1}\right).
\]

Secondly we consider those three normal factors.
The tangent to the double orbit is $\fh g_0+g_0\fh$.
We compute on each pair of block positions separately.

On the identity summand (positions $1,4$), an orbit tangent is
\[
 \begin{pmatrix}U&0\\0&V\end{pmatrix}I_{2a+\Delta}
 -I_{2a+\Delta}\begin{pmatrix}S&0\\0&T\end{pmatrix}
 =\begin{pmatrix}U-S&0\\0&V-T\end{pmatrix},
\]
where $U,S\in\Mat_{a+\Delta}(\D)$ and $V,T\in\Mat_a(\D)$.
The quotient is therefore represented uniquely by
$\left(\begin{smallmatrix}0&X_0\\Y_0&0\end{smallmatrix}\right)$,
and the stabilizer acts by
\[
 \begin{pmatrix}u_0&0\\0&v_0\end{pmatrix}
 \begin{pmatrix}0&X_0\\Y_0&0\end{pmatrix}
 \begin{pmatrix}u_0^{-1}&0\\0&v_0^{-1}\end{pmatrix}
 =\begin{pmatrix}0&u_0X_0v_0^{-1}\\v_0Y_0u_0^{-1}&0\end{pmatrix}.
\]

On the exchange summand (positions $2,5$), the tangent calculation is
\[
 \begin{pmatrix}U&0\\0&V\end{pmatrix}
 \begin{pmatrix}0&I_b\\I_b&0\end{pmatrix}
 -\begin{pmatrix}0&I_b\\I_b&0\end{pmatrix}
 \begin{pmatrix}S&0\\0&T\end{pmatrix}
 =\begin{pmatrix}0&U-T\\V-S&0\end{pmatrix},
\]
with $U,V,S,T\in\Mat_b(\D)$.
Thus the quotient is represented uniquely by $\diag(X_1,Y_1)$,
with action
\[
 \begin{pmatrix}u_1&0\\0&v_1\end{pmatrix}
 \begin{pmatrix}X_1&0\\0&Y_1\end{pmatrix}
 \begin{pmatrix}v_1^{-1}&0\\0&u_1^{-1}\end{pmatrix}
 =\begin{pmatrix}u_1X_1v_1^{-1}&0\\0&v_1Y_1u_1^{-1}\end{pmatrix}.
\]

On the generic summand (positions $3,6$), we compute the quotient
of $\Mat_{2k}(\D)$ by the tangent to the orbit of
$g_C=\left(\begin{smallmatrix}I_k&C\\I_k&I_k\end{smallmatrix}\right)$.
Two matrices represent the same normal class precisely when their
difference is an orbit tangent. Such tangents are the matrices
\[
\begin{pmatrix}U&0\\0&V\end{pmatrix}g_C-
 g_C\begin{pmatrix}S&0\\0&T\end{pmatrix}
 =\begin{pmatrix}U-S&UC-CT\\V-S&V-T\end{pmatrix},
\]
where $U,V,S,T$ vary independently in $\Mat_k(\D)$.

First, we eliminate three of the four blocks.
Let $M=\left(\begin{smallmatrix}M_{11}&M_{12}\\M_{21}&M_{22}\end{smallmatrix}\right)$
be arbitrary. Taking $U=M_{11}$, $S=0$, $V=M_{21}$, and
$T=M_{21}-M_{22}$ in the preceding formula gives the orbit tangent
\[
 \begin{pmatrix}
 M_{11}&M_{11}C-CM_{21}+CM_{22}\\
 M_{21}&M_{22}
 \end{pmatrix}.
\]
Subtracting this matrix from $M$ shows that
\[
 M\equiv
 \begin{pmatrix}
 0&M_{12}-M_{11}C+CM_{21}-CM_{22}\\
 0&0
 \end{pmatrix}
 \pmod{\text{orbit tangent space}}.
\]
Thus every normal class has a representative supported in the
upper-right block, but this representative need not yet be unique.

We next determine when two such representatives define the same class.
Their difference is an orbit tangent exactly when
\[
 \begin{pmatrix}0&Z-Z'\\0&0\end{pmatrix}
 =\begin{pmatrix}U-S&UC-CT\\V-S&V-T\end{pmatrix}
\]
for some $U,V,S,T$. The three zero blocks force
$U=S$, $V=S$, and $V=T$, hence $U=V=S=T$.
The remaining equation is therefore $Z-Z'=UC-CU$.
Conversely, every such commutator occurs by taking $U=V=S=T$.
Consequently, the normal quotient is identified with
$\Mat_k(\D)/[\Mat_k(\D),C]$ by the map
\begin{equation}\label{eq:generic-quotient}
M\longmapsto M_{12}-M_{11}C+CM_{21}-CM_{22}
 \pmod{[\Mat_k(\D),C]}.
\end{equation}

Finally, semisimplicity of $C$ identifies this quotient with $A_C$.
Namely,
\[
 \Mat_k(\D)=A_C\oplus[\Mat_k(\D),C].
\]
Every normal class is therefore represented uniquely by
$\left(\begin{smallmatrix}0&Z\\0&0\end{smallmatrix}\right)$ with
$Z\in A_C$. Since $c\in A_C^\times$, conjugation by $c$ preserves
$A_C$, and the stabilizer action on these representatives is
\[
 \begin{pmatrix}c&0\\0&c\end{pmatrix}
 \begin{pmatrix}0&Z\\0&0\end{pmatrix}
 \begin{pmatrix}c^{-1}&0\\0&c^{-1}\end{pmatrix}
 =\begin{pmatrix}0&cZc^{-1}\\0&0\end{pmatrix}.
\]
For the full normal space,
it remains to check that these three complements exhaust the normal space.
Projecting an orbit tangent onto any of the three summands gives
an orbit tangent for the corresponding restricted action.
Hence the three complements above map injectively into $N_{g_0}$.
The stabilizer calculation gives
\[
 \dim_F N_{g_0}=\dim G-\dim L+\dim L_{g_0}
 =d^2\bigl(2a(a+\Delta)+2b^2\bigr)+\dim_F A_C.
\]
This is the sum of their dimensions, so they give the full normal
space \eqref{eq:group-normal}.
\end{proof}

\subsection{The stabilizing adjoint}
Lemma~\ref{lem:rational-similarity} gives $R\in\GL_k(\D)$ such that
\[
 RC=C^\dagger R.
\]
The algebra $F[C]$ is a product of finite extensions $E_i/F$.
The centralizer theorem, applied on each primary summand, gives
\begin{equation}\label{eq:centralizer-factors}
 A_C\simeq\prod_i\Mat_{m_i}(D_i),
\end{equation}
where $D_i$ is Brauer-equivalent to $\D\otimes_FE_i$ and hence
is either $E_i$ or a quaternion division algebra over $E_i$.
The map
\begin{equation}\label{eq:centralizer-adjoint}
 \jmath_C:A_C\longrightarrow A_C,\qquad
 \jmath_C(Z)=R^{-1}Z^\dagger R
\end{equation}
is an $F$-linear anti-automorphism fixing $F[C]=Z(A_C)$ pointwise.
Thus Theorem~\ref{thm:conjugation} applies to it.

Put $h=\diag(I_{a+\Delta},I_b,R,I_a,I_b,RC)\in H$.
On the generic summand,
\[
 \begin{pmatrix}R&0\\0&RC\end{pmatrix}
 \begin{pmatrix}I_k&C\\I_k&I_k\end{pmatrix}
 \begin{pmatrix}R^{-1}&0\\0&C^{-1}R^{-1}\end{pmatrix}
 =\begin{pmatrix}I_k&I_k\\C^\dagger&I_k\end{pmatrix}=g_C^\dagger.
\]
Consequently $g_0^\dagger=hg_0h^{-1}$, and
$s_0=(h^{-1},h^{-1})s\in\widetilde L_{g_0}$ satisfies
$\chi(s_0)=-1$. Its square is the element
$(h^{-1}h^\dagger,h^{-1}h^\dagger)\in L_{g_0}$; the argument
uses $s_0$ as a stabilizing symmetry, not necessarily as an involution.

\begin{lemma}\label{lem:normal-involution}
In \eqref{eq:group-normal}, the normal action of $s_0$ is
\begin{equation}\label{eq:normal-involution}
(X_0,Y_0;X_1,Y_1;Z)
\longmapsto
(Y_0^{\dagger},X_0^{\dagger};X_1^{\dagger},Y_1^{\dagger};R^{-1}Z^{\dagger}R).
\end{equation}
\end{lemma}
\begin{proof}
The derivative is $ds_0(A)=h^{-1}A^{\dagger}h$.
On the identity and exchange summands, $h$ is the identity, so
\[
 \begin{pmatrix}0&X_0\\Y_0&0\end{pmatrix}^{\dagger}
 =\begin{pmatrix}0&Y_0^{\dagger}\\X_0^{\dagger}&0\end{pmatrix},\qquad
 \begin{pmatrix}X_1&0\\0&Y_1\end{pmatrix}^{\dagger}
 =\begin{pmatrix}X_1^{\dagger}&0\\0&Y_1^{\dagger}\end{pmatrix}.
\]
On the generic summand,
\[
\begin{aligned}
 ds_0\!\begin{pmatrix}0&Z\\0&0\end{pmatrix}
 &=\begin{pmatrix}R^{-1}&0\\0&C^{-1}R^{-1}\end{pmatrix}
   \begin{pmatrix}0&0\\Z^{\dagger}&0\end{pmatrix}
   \begin{pmatrix}R&0\\0&RC\end{pmatrix}\\
 &=\begin{pmatrix}0&0\\ C^{-1}R^{-1}Z^{\dagger}R&0\end{pmatrix}.
\end{aligned}
\]
Under \eqref{eq:generic-quotient}, the latter matrix represents
$C(C^{-1}R^{-1}Z^{\dagger}R)=R^{-1}Z^{\dagger}R$, as required.
\end{proof}

\subsection{Completion of the proof}
Now we can give the proof of Theorem \ref{thm:group}.
\begin{proof}[Proof of Theorem~\ref{thm:group}]
Let $D\in\cD(N_{g_0})^{\widetilde L_{g_0},\chi}$.
Restricting the equivariance to $L_{g_0}$, we use the decomposition in
Lemma~\ref{lem:group-normal}. Fixing test functions on the other
summands gives a distribution on each summand with the corresponding
stabilizer character.

On $\Mat_{(a+\Delta)\times a}(\D)\oplus\Mat_{a\times(a+\Delta)}(\D)$,
the character is trivial. Theorem~\ref{thm:rectangular}, with
$(p,q)=(a+\Delta,a)$, gives invariance under
$(X_0,Y_0)\mapsto(Y_0^{\dagger},X_0^{\dagger})$.

On $\Mat_b(\D)\oplus\Mat_b(\D)$, the character is
$\gamma(\Nrd u_1/\Nrd v_1)$. Corollary~\ref{cor:balanced-recip},
with $r=b$, gives invariance under
$(X_1,Y_1)\mapsto(X_1^{\dagger},Y_1^{\dagger})$.

On $A_C$, the character is trivial and $A_C^\times$ acts by
conjugation. Equations~\eqref{eq:centralizer-factors} and
\eqref{eq:centralizer-adjoint}, together with
Theorem~\ref{thm:conjugation}, give invariance under
$Z\mapsto R^{-1}Z^\dagger R$.

By Lemma~\ref{lem:normal-involution}, these three symmetries
combine to give $(s_0)_*D=D$.
On the other hand, $\widetilde L_{g_0}$-equivariance gives
$(s_0)_*D=\chi(s_0)D=-D$, so $D=0$.
Thus $\cD(N_{g_0})^{\widetilde L_{g_0},\chi}=0$ for every
representative supplied by Lemma~\ref{lem:closed-orbits}.
Theorem~\ref{thm:HC-descent} now yields
\[
\cD(G)^{\widetilde L,\chi}=0.
\]
For a distribution satisfying \eqref{eq:biequiv}, its adjoint-odd
part lies in this space by \eqref{eq:s-normalization}; hence it is zero. This finishes the proof.
\end{proof}
Theorem \ref{thm:main} follows from Theorem \ref{thm:group} easily.
\begin{proof}[Proof of Theorem~\ref{thm:main}]
	Let $\pi$ be an irreducible smooth representation of $G$.
	The Gelfand--Kazhdan criterion (see \cite{GK} or \cite[Theorem 2.21]{JLX}) asserts that adjoint
	invariance of all distributions satisfying \eqref{eq:biequiv} implies
	\[
	\dim\Hom_H(\pi,\eta)\,
	\dim\Hom_H(\pi^\vee,\eta^{-1})\leq1.
	\]
	Theorem~\ref{thm:group} supplies this invariance.
	Corollary~\ref{cor:contragredient} gives
	$\pi^\vee\simeq\pi\circ\theta$, where $\theta(g)=(g^{\dagger})^{-1}$. Since $\theta(H)=H$ and
	$\eta(\theta(h))=\eta(h)^{-1}$, we obtain
	\[
	\Hom_H(\pi^\vee,\eta^{-1})\simeq\Hom_H(\pi,\eta).
	\]
	Hence $\dim\Hom_H(\pi,\eta)^2\leq1$, which proves the theorem.
\end{proof}
\begin{remark} In order to end this paper,
	we would like to highlight the fact that this method does not work for the case when $d=\deg_F \D\geq 3$ since there is no involution taking $\pi$ to its contragredient $\pi^\vee$ (see \cite{Ragh}). 
\end{remark}

\end{document}